\documentclass{article}

\usepackage{amsmath}
\usepackage{amsthm}
\usepackage{amsfonts}
\usepackage{mathtools}
\usepackage[numbers]{natbib}
\usepackage{xcolor}
\usepackage{caption} 

\usepackage{hyperref}
\hypersetup{colorlinks=true, linkcolor=blue, citecolor=blue}

\newcommand{\R}{\mathbb{R}}

\newcommand{\cL}{\mathcal{L}}
\newcommand{\cA}{\mathcal{A}}
\newcommand{\cS}{\mathcal{S}}
\newcommand{\cO}{\mathcal{O}}

\usepackage{bm}
\newcommand{\xb}{\bm{x}}

\newcommand{\Ab}{\bm{A}}
\newcommand{\Bb}{\bm{B}}
\newcommand{\Eb}{\bm{E}}
\newcommand{\Lb}{\bm{L}}
\newcommand{\Mb}{\bm{M}}
\newcommand{\Pb}{\bm{P}}

\newcommand{\cLtau}{\cL_{\tau}}
\newcommand{\cStau}{\cS_{\tau}}

\DeclareMathOperator{\Id}{Id}

\renewcommand{\exp}[1]{\mathrm{e}^{#1}}

\newcommand{\domain}[1]{\mathcal{D}(#1)}

\newcommand{\eps}{\varepsilon}
\renewcommand{\phi}{\varphi}

\DeclarePairedDelimiter{\iprod}{(}{)}

\DeclarePairedDelimiterXPP\lnorm[1]{}\lVert\rVert{_2}{#1}

\newcommand\MTkillspecial[1]{
    \bgroup
    \catcode`\&=9
    \let\\\relax%
    \scantokens{#1}%
    \egroup
    }
\DeclarePairedDelimiter\norm{\lVert}{\rVert}
\reDeclarePairedDelimiterInnerWrapper\norm{star}{
\mathopen{#1\vphantom{\MTkillspecial{#2}}\kern-\nulldelimiterspace\right.}
#2
\mathclose{\left.\kern-\nulldelimiterspace\vphantom{\MTkillspecial{#2}}#3}}

\DeclarePairedDelimiter\normLHH{\lVert}{\rVert_{\cL(H)}}
\DeclarePairedDelimiter\maxnorm{\lVert}{\rVert_{\infty}}

\newcommand{\diff}[1]{\mathrm{d}#1}

\newcommand{\dt}{\diff{t}}
\newcommand{\ds}{\diff{s}}
\newcommand{\dr}{\diff{r}}

\newcommand{\ddx}[1]{\frac{\mathrm{d}}{\diff{#1}}}
\newcommand{\ddjx}[2]{\frac{\mathrm{d}^{#1}}{\diff{#2}^{#1}}}
\newcommand{\dddx}[1]{\ddjx{2}{#1}}
\newcommand{\ddddx}[1]{\ddjx{3}{#1}}

\newcommand{\ddr}{\ddx{r}}
\newcommand{\dds}{\ddx{s}}
\newcommand{\dddr}{\dddx{r}}
\newcommand{\ddds}{\dddx{s}}

\newcommand{\dddds}{\ddddx{s}}

\newcommand{\SH}{\Sigma(H)}
\newcommand{\SHP}{\Sigma^+(H)}

\newcommand{\CpH}[1]{C_p\big(#1, \SH\big)}
\newcommand{\CpHk}[2]{C_p^{#1}\big(#2, \SH\big)}

\newcommand{\CHk}[2]{C^{#1}\big(#2, \SH\big)}

\numberwithin{equation}{section}
\theoremstyle{plain}
\newtheorem{theorem}{Theorem}[section]
\newtheorem{lemma}[theorem]{Lemma}
\newtheorem{corollary}[theorem]{Corollary}
\newtheorem{assumption}{Assumption}

\theoremstyle{definition}
\newtheorem{remark}{Remark}[section]

\title{Convergence analysis of Lie and Strang splitting for operator-valued differential Riccati equations}

\author{Eskil Hansen \and Tony Stillfjord \and Teodor \r{A}berg}

\begin{document}

\maketitle

\begin{abstract}
  Differential Riccati equations (DREs) are semilinear matrix- or operator-valued differential equations with quadratic non-linearities. They arise in many different areas, and are particularly important in optimal control of linear quadratic regulators, where they provide the optimal feedback control laws. In the context of control of partial differential equations, these Riccati equations are operator-valued. To approximate their solutions, both spatial and temporal discretizations are needed. While the former have been well analyzed in the literature, there are very few rigorous convergence analyses of time stepping methods applied to DREs, particularly in the infinite-dimensional, operator-valued setting. 
  In view of this, we analyze two numerical time-stepping schemes, the Lie and Strang splitting methods, in such a setting. The analysis relies on the assumption that the uncontrolled system evolves via an operator that generates an analytic semigroup, and that either the initial condition is sufficiently smooth, or the nonlinearity in the DRE is sufficiently smoothing. These assumptions are mild, in the sense that they are not enough to even guarantee continuity in operator-norm of the exact solution to the DRE. However, they imply certain regularity in a pointwise sense, which can be leveraged to prove convergence in operator-norm with the classical orders.
  The results are illustrated by four numerical experiments, where convergence with the expected order is correlated with the relevant assumptions being fulfilled.
  The experiments also demonstrate that matrix-valued DREs which arise as spatial discretizations of operator-valued DREs behave similarly, unless the discretization is coarse.
\end{abstract}

\section{Introduction} \label{sec:introduction}
A differential Riccati equation (DRE) is of the form
\begin{equation}\label{eq:DRE}
  \dot{P}(t) = A^* P(t) + P(t)A + Q - P(t)SP(t),\quad t\in (0,T],\quad P(0) = P_0,
\end{equation}
where the given data $A$, $Q$ and $S$ and the solution values $P(t)$ are either real matrices or linear operators on a Hilbert space. By $A^*$ we mean either the adjoint of $A$ if it is an operator, or its Hermitian transpose if it is a matrix. 

A typical application where DREs arise is the linear quadratic regulator (LQR) problem, where one aims to steer the state vector $x(t)$ to a desirable state by adjusting an input function $u(t)$. The state $x$ satisfies
\begin{align*}
  \dot{x} &= Ax + Bu, \\
  y &= Ex,
\end{align*}
with given matrices or operators $A$, $B$ and $E$, and the typical goal is to minimize the cost functional
\begin{equation*}
J(u) = \int_{0}^{T}{\|y(t)\|^2 + \|u(t)\|^2 \, \dt} + \norm{Gx(T)}^2,
\end{equation*}
where  $G$ is a penalization operator for the final state. It is well-known that the optimal feedback control strategy for doing so is $u(t) = -B^*P(T-t)x(t)$, where $P$ satisfies the DRE~\eqref{eq:DRE} with $Q = E^*E$, $S = BB^*$ and $P_0 = G^*G$.

In the finite-dimensional LQR case, $x(t) \in \R^N$ for some $N$ and the above constitutes control of a system of $N$ ordinary differential equations (ODEs). This leads to a matrix-valued DRE. If $x(t)$ is instead a function in, e.g., $L^2(\Omega)$, where $\Omega$ is the spatial domain, we are concerned with control of a partial differential equation (PDE) and the corresponding DRE is operator-valued. Often, a system of ODEs arises as a spatial discretization of a PDE. In this case, when the discretization is refined, the solution to the matrix-valued DRE should tend to the solution of the corresponding operator-valued DRE.

In this paper, we analyze two time-stepping methods applied to operator-valued DREs. As indicated above, there are two main reasons for focusing on operator-valued rather than matrix-valued equations. First, it allows us to ignore the spatial discretization, and thereby decouple the temporal and spatial errors. A full discretization can be acquired by subsequently discretizing the time-discrete system in space. Secondly, given that finer spatial discretizations typically result in systems with less nice properties, the operator-valued DRE also corresponds to the ``worst'' possible spatial discretization in terms of analysis. Our temporal convergence results will automatically hold for the spatially discretized matrix-valued equations under minimal assumptions, regardless of how fine or coarse these discretizations are.

As our numerical experiments show, the fact that a time-stepping method is convergent for a DRE in $R^{N \times N}$ for any $N$ does not imply that it is convergent in the limit as $N \to \infty$. Basically, the error constants may depend on $N$ in non-obvious ways, or on regularity properties that only become apparent for large $N$. This means that a large $N$ might require unfeasibly small time steps before any decrease in the temporal errors is observed, unless the problem satisfies certain properties. Our operator-valued analysis exposes these properties.

The time-stepping methods that we are concerned with here are the Lie and Strang splitting schemes. Instead of approximating the solution to the full DRE~\eqref{eq:DRE} directly, these methods work by combining approximations of the solutions to the two subproblems $\dot{P} = FP$ and $\dot{P} = GP$, where, informally, 
\begin{align*}
  F P &=  A^*P + PA + Q, \\
  G P &= -PSP.
\end{align*}
 We denote the corresponding solution operators by $\exp{tF}$ and $\exp{tG}$, respectively. Then with the approximation $P^n \approx P(n\tau)$, the Lie splitting method is given by
 \begin{equation*}
   P^{n+1} = \exp{\tau F}\exp{\tau G} P^n,
 \end{equation*}
and the Strang splitting is defined by
 \begin{equation*}
   P^{n+1} = \exp{\tau/2F}\exp{\tau G}\exp{\tau/2F} P^n.
 \end{equation*}
 How to properly interpret and define the operators $F$ and $G$, as well as the solution operators $\exp{tF}$ and $\exp{tG}$ will be covered in Section~\ref{sec:setting}.

 Many other time-stepping methods for (matrix-valued) DRE have been suggested in the literature, e.g.\ BDF methods \cite{BennerMena2004}, Rosenbrock methods~\cite{BennerMena2013}, Peer methods~\cite{BennerLang2018}, and exponential integrators~\cite{LiZhangLiu2021}. The splitting schemes considered in this paper were introduced in~\cite{Stillfjord2015}, see also~\cite{Stillfjord2017}. The same methods but with a different decomposition of the right-hand-side were considered in~\cite{MenaOstermannPfurtschellerPiazzola2018,OstermannPiazzolaWalach2018}. A common, but surprising, theme in all these works is a lack of rigorous convergence analysis. We are only aware of three exceptions: First, \cite{HansenStillfjord2014} considers an operator-valued setting, but the analyzed method does not fully take advantage of the structure of the DRE and is therefore less efficient than the methods we consider here. Secondly, the methods considered in~\cite{OstermannPiazzolaWalach2018} are similar to ours, but uses a different decomposition of the right-hand-side that typically leads to larger errors for the same step size. Additionally, the analysis is only for the matrix-valued case. Finally, the unpublished manuscript~\cite{Stillfjord2015b} considers the same methods as in this paper in an operator-valued setting. However, this particular setting is too restrictive and the analysis is therefore only applicable to very few problems.
 The main goal of the present paper is thus to remedy this situation and provide a rigorous convergence analysis for Lie and Strang splitting that is applicable also in an operator-valued setting.

We note that spatial discretizations of operator-valued DREs have been analyzed to a much greater extent than temporal discretizations. Convergence without order is shown in e.g.~\cite{Rosen1991,BennerMena2016}. As far as we are aware, the first proof of convergence with almost optimal orders in a setting similar to ours is given in~\cite{KrollerKunisch1991}. The comprehensive book~\cite{LasieckaTriggiani2000} proves convergence with orders also in more general settings, but only for \emph{algebraic} Riccati equations. These results can, however, likely be extended to DREs.

Informally, the main assumptions on~\eqref{eq:DRE} in this paper is that $A$ generates an analytic semigroup (see, e.g., \cite[Chapter 2.5]{Pazy1983}), and that $B$ and $E$ are bounded operators. This allows for applications such as distributed control. Given this, our main result is that Lie and Strang splitting are convergent with the expected orders $1$ and $2$, respectively, under certain regularity assumptions. For Lie, we (again informally) require either that the initial condition $P_0$ is in the domain of $P \mapsto A^*P + PA$, or that $S$ satisfies a similar condition. The former condition is less restrictive than the usual one, where the operator would be applied twice. This is possible due to the smoothing properties of $A$. In the latter case, $P_0$ can be arbitrarily non-smooth, but the smoothness assumption on $S$ means that the solution $P(t)$ is nevertheless sufficiently smooth. For Strang splitting, we require both the above assumptions.

A main difficulty here is that the above-mentioned regularity only holds away from $t = 0$. The solution $t \mapsto P(t)$ is not differentiable or even continuous at $t = 0$ in operator-norm, even under the above assumptions. Thus the usual Taylor expansions which could be done in the matrix-valued case are no longer valid. However, the assumptions guarantee that for any $x$, $t \mapsto P(t)x$ is either continuous or continuously differentiable, and $P$ or $\dot{P}$ is uniformly bounded. This allows us to instead expand around $t=\epsilon > 0$ and by careful manipulation of the rest terms prove that they are uniformly bounded independent of $\epsilon$.

The paper is organized as follows. In Section~\ref{sec:setting}, we discuss the abstract setting and make the above informal statements mathematically rigorous. We also state and prove some preliminary lemmas on existence and smoothness of solutions of the DRE. The convergence analysis of Lie splitting is presented in Section~\ref{sec:analysis_Lie}, and followed by the analysis of Strang splitting in Section~\ref{sec:analysis_Strang}. These results are illustrated by three numerical experiments in Section~\ref{sec:experiments}, where the different assumptions are either satisfied or not. By repeatedly refining a spatial discretization, we approach the operator-valued case. Finally, we present our conclusions and ideas for future work in Section~\ref{sec:conclusions}.

\section{Setting}\label{sec:setting}
Given a Hilbert space $X$, we use $\iprod{\cdot,\cdot}_X$ to denote its inner product, or simply $\iprod{\cdot,\cdot}$ when the particular space is clear from the context. The induced operator norm is denoted by $\norm{\cdot}_X$ or simply $\norm{\cdot}$.
Given two general Hilbert spaces $X$ and $Y$, we denote by $\cL(X,Y)$ the space of linear bounded operators from $X$ to $Y$. If $X=Y$, we abbreviate $\cL(X) = \cL(X,X)$
If $L: X \to Y$ then $L^*: Y \to X$ denotes the Hilbert space adjoint of $L$ that satisfies $\iprod{Lx, y}_Y = \iprod{x, L^*y}_X$ for all $x \in X$ and $y \in Y$.

Now let $H$, $U$ and $Y$ be specific given Hilbert spaces, with $H$ complex. Since the solution $P(t)$ should be symmetric and positive definite, we will consider the Banach space $\SH = \{P \in \cL(H) \;|\; P = P^*\}$ with the $\cL(H)$-norm and the cone $\SHP = \{P \in \SH \;|\; P\ge 0 \}$.
We make the following assumption on the operators $A$, $Q$ and $S$ in~\eqref{eq:DRE}:

\begin{assumption}\label{ass:operators}
  There are operators $B \in \cL(U,H)$ and $E \in \cL(H,Y)$ such that $Q$ and $S$ can be factorized as $Q=E^*E$ and $S=BB^*$. Further, $P_0 \in \SHP$, and the unbounded operator $A: \domain{A} \subset H \to H$ is the generator of an analytic semigroup $t \mapsto \exp{tA} \in \cL(H)$.
\end{assumption}

\begin{remark}\label{rem:basic_operator_properties}
  It follows directly from the assumption that $Q, S \in \SHP$. Further, $A$ is closed and $\domain{A}$ is dense in $H$~\cite[Corollary 1.2.5]{Pazy1983}.
\end{remark}

\begin{remark}
  In the context of LQR problems, we assume that the control operator $B$ and the observation operator $E$ are both bounded. This allows for distributed control and observation but typically not boundary control and observation. Extending the work to less regular $B$ and $E$ is likely possible, but at the cost of reduced convergence orders and significantly more technical machinery.
\end{remark}

In the matrix-valued case when $P(t)$ and $A$ are in $\R^{N \times N}$, the right-hand side of~\eqref{eq:DRE} is well defined. In our context, we are interested in operator-valued problems and search for a solution $P$ with $P(t) \in \SHP \subset \SH$. Under Assumption~\ref{ass:operators}, the terms $Q$ and $P(t)SP(t)$ are still well defined as operators in $\SH$. However, since $A$ is unbounded we need to properly interpret $A^*P(t) + P(t)A$, which is not necessarily in $\SH$. We follow~\cite{Bensoussan2007} and define $\cA: \domain{\cA} \subset \SH \to \SH$ by
\begin{equation*}
  \iprod{\cA(P)x, y} = \phi_P(x, y), \quad x,y \in \domain{A}, P \in \domain{\cA}, 
\end{equation*}
where
\begin{equation*}
  \phi_P(x, y) = \iprod{Px,Ay} + \iprod{Ax,Py}, \quad x,y \in \domain{A},
\end{equation*}
and
\begin{equation*}
  \domain{\cA} = \{P \in \SH \;;\; \phi_P \text{ is continuous in } H \times H\}.
\end{equation*}
Essentially this means that $\domain{\cA}$ are those $P \in \SH$ for which $A^*P + PA$ can be extended to an operator in $\SH$. On $\domain{A}$, this extension agrees with the original expression:
\begin{lemma}[Proposition 3.1~\cite{Bensoussan2007}] \label{lemma:cAidentification}
  If $P \in \domain{\cA}$ then $x \in \domain{A}$ implies that $Px \in \domain{A^*}$ and
  \begin{equation*}
    \cA Px = A^*Px + PAx.
  \end{equation*}
\end{lemma}

With this in mind, for the rest of the paper we will be considering the equation
\begin{equation} \label{eq:DREclassical}
  \dot{P} = \cA P + Q - PSP
\end{equation}
rather than~\eqref{eq:DRE}. According to Lemma~\ref{lemma:cAidentification}, if $P(t) \in \domain{\cA}$ then these equations coincide after multiplication with $x \in \domain{A}$.

The operator $A$ is a generator of an analytic semigroup $\exp{tA}$, and this gives rise to an analytic semigroup $\exp{t\cA}: \SH \to \SH$ generated by $\cA$, defined by
\begin{equation*}
  \exp{t\cA}P = \exp{tA^*} P \exp{tA}, \quad t \ge 0.
\end{equation*}
As noted in~\cite[Remark 3.1]{Bensoussan2007}, it is not strongly continuous, i.e.\  $\normLHH{\exp{t\cA}P - P} \not\to 0$ as $t \to 0^+$. However, it satisfies $\norm{\exp{t\cA}Px - Px} \to 0$ as $t \to 0^+$ for every $x \in H$ and $P \in \SH$.

In the following, throughout the paper, $C$ denotes an arbitrary constant which may be different from line to line. It will often depend on the problem data $T$, $\normLHH{Q}$ and $\normLHH{S}$, but for brevity we only specify its dependence on other parameters.
\begin{lemma}\label{lemma:analyticSG}
 There exist a constant $C$ such that the analytic semigroup $\exp{tA}$ generated by the operator $A$ satisfies the inequalities
  \begin{equation*}
    \normLHH{\exp{tA}} \le C \quad \text{and} \quad \normLHH{A\exp{tA}} \le \frac{C}{t}
  \end{equation*}
  for $t \in (0,T]$. Similarly, the analytic semigroup $\exp{t\cA}$ satisfies
  \begin{equation*}
    \normLHH{\exp{t\cA}P} \le C \normLHH{P} \quad \text{and} \quad \normLHH{\cA\exp{t\cA}P} \le \frac{C}{t}\normLHH{P}
  \end{equation*}
   for $t \in (0,T]$ and $P \in \SH$.
\end{lemma}
\begin{proof}
  The first two inequalities are standard results for analytic semigroups, see e.g.~\cite[Lemma 12.32]{RenardyRogers2004}. They are usually stated with $C\exp{\omega t}$ instead of $C$, where $\omega \in \R$, but the exponential is clearly bounded by $\max (1, \exp{\omega T})$ on the finite time interval $[0,T]$.
  
  Since $A$ is closed (Remark~\ref{rem:basic_operator_properties}), it holds that $(A^*)^* = A$~\cite[Theorem 8.5.7]{RenardyRogers2004} and thus $\bigl ( \exp{tA^*} \bigr)^* = \exp{tA}$~\cite[Corollary 1.10.6]{Pazy1983}.
  This means that the bound for $\exp{t\cA}$ follows directly by applying the 
  bounds of the individual semigroups $\exp{tA}$ and $\exp{tA^*}$. 
    We next verify that $\exp{t\cA}P$ lies in $\domain{\cA}$. This holds, since for any $t \in (0,T]$ and $x,y \in \domain{A}$,
    \begin{equation*}
    \begin{split}
        \vert \phi_{\exp{t\cA}P}(x,y)\vert 
        &= \vert \iprod{A x, \exp{tA^*}P\exp{tA}y} + \iprod{\exp{tA^*}P\exp{tA}x, Ay} \vert  \\
        & = \vert \iprod{ \exp{tA}Ax, P \exp{tA}y} + \iprod{P\exp{tA}x, \exp{tA}Ay}\vert \\
        & = \vert \iprod{ A\exp{tA}x, P \exp{tA}y} + \iprod{P\exp{tA}x, A\exp{tA}y}\vert \\
        & \leq \normLHH{A \exp{tA}}\norm{x}\normLHH{P}\normLHH{\exp{tA}}\norm{y} \\
        & \quad + \normLHH{A \exp{tA}}\norm{x}\normLHH{P}\normLHH{\exp{tA}}\norm{y} \\
        &\leq \frac{C}{t} \normLHH{P} \norm{x} \norm{y}.
    \end{split}
  \end{equation*}
As noted in Remark~\ref{rem:basic_operator_properties}, $\domain{A}$ is dense in $H$, and the bound thus holds for all $x, y \in H$. But, by definition,
\begin{equation*}
  \normLHH{\cA \exp{t\cA} P}
  = \sup_{x\in H} \sup_{y\in H} \frac{\vert \iprod{\cA(P)x, y} \vert}{\norm{x}\norm{y}} 
  =\sup_{x\in H} \sup_{y\in H} \frac{\vert \phi_{\exp{t\cA}P}(x,y) \vert}{\norm{x}\norm{y}},
\end{equation*}
so the above inequality finishes the proof.
\end{proof}

Our basic assumptions on the operators agree with those in~\cite{Bensoussan2007}, and imply the existence of a unique solution to the problem:
\begin{lemma}[{\cite[Theorem IV-1.3.1]{Bensoussan2007}}] 
  \label{lemma:existence}
  Under Assumption~\ref{ass:operators}, there is a (classical) solution $P$ to~\eqref{eq:DREclassical} on $t \in [0, T]$. That is, $P(0) = P_0$ and for any $t \in (0,T)$, $P(t)$ is differentiable with $P(t) \in \domain{\cA} \cap \SHP$ and $\dot{P}(t) = \cA P(t) + Q - P(t)SP(t)$.
  Further, for any $\eps > 0$, $P \in \CHk{\infty}{[\eps, T]}$.
\end{lemma}
\noindent Here,
\begin{equation*}
  \CHk{\infty}{[\eps, T]} =
  \left\{
  \begin{aligned}
  &t \mapsto P(t) \in \SH \text{ is infinitely differentiable in } \cL(H) \\  
  &\text{for } t\in [\eps,T] \text{ and }\sup_{t\in (\eps,T)}\normLHH{\ddjx{k}{t}P(t)}< \infty, \; k = 0, 1, \ldots
\end{aligned}
\right\}.
\end{equation*}
Thus a solution exists regardless of how smooth $P_0$ is, and it is infinitely smooth away from $t = 0$. However, the solution will typically not be continuous at $t=0^+$ in the $\cL(H)$-norm. We can only guarantee that $P$ is strongly continuous. 
We therefore introduce the function space
\begin{equation*}
  \CpH{[0, T]}=
  \left\{
  \begin{aligned}
  t \mapsto P(t) \in \SH \;\vert\; &t \mapsto P(t)x \text{ is continuous in } H  \\  
  &\text{for any } x \in H
\end{aligned}
\right\},
\end{equation*}
equipped with the norm 
\begin{equation*}
  \maxnorm{P} = \sup_{t\in (0,T)}\normLHH{P(t)}.
\end{equation*}
As noted in~\cite[Section IV-1.2.1]{Bensoussan2007}, $\maxnorm{P}$ is finite for any $P \in \CpH{[0, T]}$ by the uniform boundedness theorem.

We also consider the spaces $\CpHk{k}{[0, T]}$ which consist of functions $t \mapsto P(t) \in \SH$ such that $t \mapsto P(t)x$ is $k$ times differentiable for any $x \in H$ and where the $k$-th derivatives belongs to $\CpH{[0, T]}$. 
With these definitions, we have the following regularity near $t=0$:
\begin{lemma}[{\cite[Proposition IV-1.3.2 and Theorem IV-1.3.1]{Bensoussan2007}}] 
  \label{lemma:regularity}
  Under Assumption~\ref{ass:operators} the solution $P$ guaranteed to exist by Lemma~\ref{lemma:existence} satisfies $P \in \CpH{[0, T]}$. If in addition $P_0 \in \domain{\cA}$, then $P \in \CpHk{1}{[0, T]}$ and for any $t \in [0,T]$, $\cA P \in \CpH{[0, T]}$.
\end{lemma}
\begin{remark}
  In~\cite{Bensoussan2007}, the function spaces $C_s([0,T], \SH)$ and $C_u([0,T], \SH)$ are considered, and the above results are stated as $P \in C_s([0,T], \SH)$. However, these two spaces contain the same sets of functions, and the results can thus equivalently be stated for $P \in C_u([0,T], \SH)$. We prefer to use this version here because $C_u([0,T], \SH)$ is a Banach space, in contrast to $C_s([0,T], \SH)$ which is equipped with a weaker topology. Our space $\CpH{[0, T]}$ is the same as $C_u([0,T], \SH)$ but has been renamed in an attempt to minimize confusion. The subscript $p$ is meant to be read as ``pointwise'' (for points $x \in H$).
\end{remark}

We also have positive semi-definiteness:
\begin{lemma}[{\cite[Proposition IV-1.3.3 and Theorem IV-1.2.1]{Bensoussan2007}}] \label{lemma:solution_SHP}
  For each $t \in [0,T]$, the solution $P$ to~\eqref{eq:DREclassical} guaranteed by Lemma~\ref{lemma:regularity} satisfies $P(t) \in \SHP$.
\end{lemma}

We now define the operators $F: \domain{\cA} \to \cL(H)$ and $G: \cL(H) \to \cL(H)$ by
\begin{align} \label{eq:FG}
  F P &= \cA P + Q,\\
  G P &= - PSP,
\end{align}
and consider the two subproblems of~\eqref{eq:DREclassical} that arise when we either omit the $G P$ part or the $F P$ part of the equation:
\begin{corollary} \label{cor:subproblem_solutions}
  Under Assumption~\ref{ass:operators} there exist unique solutions to the two subproblems
  \begin{equation} \label{eq:subproblem_F}
    \dot{P} = F P, \quad P(0) = P_0
  \end{equation}
  and
  \begin{equation} \label{eq:subproblem_G}
    \dot{P} = GP, \quad P(0) = P_0,
  \end{equation}
  which satisfy the stated regularity properties in Lemma~\ref{lemma:existence} and Lemma~\ref{lemma:regularity}. For the second subproblem, the conclusions of Lemma~\ref{lemma:regularity} hold also without $P_0 \in \domain{\cA}$.
\end{corollary}
\begin{proof}
The subproblems can be seen as instances of~\eqref{eq:DREclassical}, where we either take $S = 0$ or we take $A = Q = 0$. These choices certainly fulfil the respective parts of Assumption~\ref{ass:operators}, so we can apply the previous lemmas to establish the existence and regularity properties. The final assertion follows simply by the fact that $\domain{\cA} = \SH$ if $A = 0$.
\end{proof}
We will frequently use the notation $t \mapsto \exp{tF}P_0$ and $t \mapsto \exp{tG}P_0$ to refer to the solutions to the subproblems~\eqref{eq:subproblem_F} and~\eqref{eq:subproblem_G}. In the case of~\eqref{eq:subproblem_F} we also have the representation 
\begin{equation*}
\exp{tF}P_0=\exp{t\cA}P_0+\int_0^t\exp{(t-s)\cA}Q\,\ds,
\end{equation*}
and the lack of $A$ in~\eqref{eq:subproblem_G} allows us to say more than Corollary~\ref{cor:subproblem_solutions}:
\begin{lemma}\label{lemma:NlinFlow}
   Under Assumption~\ref{ass:operators}, the solution to the problem~\eqref{eq:subproblem_G} is given by $\exp{tG}P_0 = (I + tP_0S)^{-1}P_0$, where
     \begin{equation*}
      \maxnorm{t\mapsto\exp{tG}P_0} \leq \big(1 + T \rho \normLHH{S}\big) \rho
    \end{equation*}
    for every $P_0$ with $\norm{P_0} \leq \rho$.
    Additionally, the derivatives of $t\mapsto\exp{tG}P_0$ are bounded by
    \begin{equation*}
     \maxnorm[\Big]{t\mapsto\ddjx{j}{t} \exp{tG}P_0} \le j! \normLHH{S}^j  \big(1 + T \rho \normLHH{S}\big)^{j+1} \rho^{j+1}
    \end{equation*}
    for $j = 1, 2, \ldots$.
  \end{lemma}
  \begin{proof}
    We first note that for any $P_1, P_2 \in \SHP$, we have
\begin{equation*}
  \normLHH{(I + P_1 P_2)^{-1}} \le 1 + \normLHH{P_1} \normLHH{P_2} 
\end{equation*}
see e.g.~\cite[Lemma 2A.1]{LasieckaTriggiani2000}. Thus the function $P(t)$ given by
\begin{equation*}
  P(t) = (I + tP_0 S)^{-1} P_0
\end{equation*}
is well defined for all $t \ge 0$, and $\normLHH{P(t)} \leq \big(1 + t \normLHH{P_0} \normLHH{S}\big) \normLHH{P_0}$. As a consequence, the first inequality is satisfied.
To prove that $P$ solves the problem~\eqref{eq:subproblem_G}, we first note that
\begin{align*}
  &\normLHH{(I + (t+\tau)P_0  S)^{-1}  P_0 - (I + tP_0  S)^{-1}  P_0} =\\
&\quad= \normLHH[\Big]{\big(I + (t+\tau)P_0S\big)^{-1}\Big[ (I + tP_0S) - \big(I+(t+\tau)P_0S \big) \Big] (I + tP_0S)^{-1}P_0 ) } \\
                         &\quad= \normLHH[\big]{-\tau \big( I + (t+\tau)P_0S \big)^{-1} P_0S (I + tP_0S)^{-1}P_0 } \\
                         &\quad\le \tau\normLHH{(I + (t+\tau)P_0S)^{-1}} \normLHH{P_0S} \normLHH{(I + tP_0S)^{-1}} \normLHH{P_0S}  \\
                         &\quad\le \tau \big(1 + (t+\tau)\normLHH{P_0} \normLHH{S} \big) \big(1 + t \normLHH{P_0} \normLHH{S} \big) \normLHH{P_0}^2 \\
&\quad\le \tau \, C(t, \normLHH{P_0}, \normLHH{S}),
\end{align*}
and $t \mapsto P(t)$ is therefore continuous in $\cL(H)$. By the same construction we obtain that
\begin{equation*}
\lim_{\tau  \to 0}\normLHH{(P(t+\tau) - P(t))/\tau + P(t)SP(t) } = 0.
\end{equation*}
The function $t \mapsto P(t)$ is thus continuously differentiable and satisfies the equation, i.e., $P(t)=\exp{tG}P_0$. The bound for the higher derivatives follows by induction and the chain rule, since e.g.
\begin{align*}
  \ddjx{2}{s} \exp{sG}P_0 &= \dds\Big(-\exp{sG}P_0S\exp{sG}P_0\Big) \\
                          &= -\Big(\dds \exp{sG}P_0\Big) S \exp{sG}P_0 - \exp{sG}P_0S\Big(\dds \exp{sG}P_0\Big).
\end{align*}
\end{proof}

We also have the following explicit bounds on the function $G$ and a similar function that appears in the analysis of the Strang splitting method.
\begin{lemma}\label{lemma:LipBounds}
  The functions $G: \cL(H) \to \cL(H)$ and $M: \cL(H) \to \cL(H)$ given by 
  \begin{equation*}
    GP = -PSP\quad \text{and} \quad MP = PSPSP
  \end{equation*}
  are both locally Lipschitz continuous, and 
  \begin{align}
    \normLHH{GP_1 - GP_2} &\leq 2\rho \normLHH{S} \normLHH{P_1 - P_2} \\
    \normLHH{MP_1 - MP_2} &\leq 3\rho^2 \normLHH{S}^2 \normLHH{P_1 - P_2}
  \end{align}
  for all $P_1, P_2$ such that $\normLHH{P_i} \leq \rho$.
\end{lemma}
\begin{proof}
  Both claims follow from straightforward calculations. First,
  \begin{equation*}
    \begin{split}
      \normLHH{GP_1 &- GP_2} \\
      &\quad \leq \normLHH{P_1 S P_1 - P_1 S P_2 + P_1 S P_2 - P_2 S P_2} \\
      &\quad \leq \normLHH{P_1 S} \normLHH{P_1 - P_2} +  \normLHH{P_1 - P_2} \normLHH{S P_2 } \\
      &\quad \leq 2\rho \normLHH{S} \normLHH{P_1 - P_2},
    \end{split}
  \end{equation*}
  which proves the first bound, and
  \begin{equation*}
    \begin{split}
      \normLHH{MP_1 - &MP_2} \\
      &\leq \normLHH{P_1 S P_1 S P_1 - P_1 S P_2 S P_2 + P_1 S P_2 S P_2 - P_2 S P_2 S P_2 } \\
      &\leq \normLHH{P_1 S} \normLHH{GP_1 - GP_2} \\
      &\quad +  \normLHH{P_1 - P_2} \normLHH{S P_2 S P_2} \\
      &\leq 3\rho^2 \normLHH{S}^2 \normLHH{P_1 - P_2},
    \end{split}
  \end{equation*}
  which proves the second.
\end{proof}

Finally, the solution operators for both subproblems, and thus also the splitting schemes, are invariant on $\SHP$:
\begin{lemma}\label{lemma:schemes_SHP}
  Let Assumption~\ref{ass:operators} be satisfied and let $P \in \SHP$. Then for any $t \ge 0$, also $\exp{tF}P \in \SHP$ and $\exp{tG}P \in \SHP$.  
\end{lemma}
\begin{proof}
  Since
  \begin{equation*}
    \iprod{\exp{t\cA}Px, x}
    = \iprod{\exp{tA^*}P\exp{tA}x, x}
    = \iprod{P\exp{tA}x, \exp{tA}x}
    \ge 0
  \end{equation*}
  and $(\exp{tA^*}P\exp{tA})^* = \exp{tA^*}P\exp{tA}$ it follows that $\exp{t\cA}P \in \SHP$. Since $Q \in \SHP$, it then follows by the variation-of-constants formula that $\exp{tF}P \in \SHP$.

  For the nonlinear subproblem, we first note that since $P \in \SHP$, there exists an operator $P^{1/2} \in \SHP$ such that $P = P^{1/2} P^{1/2}$. By~\cite[Lemma 2A.1]{LasieckaTriggiani2000}, we have
\begin{equation*}
  (I + tPS)^{-1} = I - tP^{1/2} (I + tP^{1/2} S P^{1/2})^{-1} P^{1/2} S.
\end{equation*}
Since $\exp{tG}P = (I + tPS)^{-1}P$ by Lemma~\ref{lemma:NlinFlow}, this means that 
\begin{align*}
  \exp{tG}P &= P^{1/2} P^{1/2} - P^{1/2} \Big( (I + tP^{1/2} S P^{1/2})^{-1} t P^{1/2} S P^{1/2} \Big) P^{1/2} \\
                 &= P^{1/2} (I + tP^{1/2} S P^{1/2})^{-1} P^{1/2}.
\end{align*}
This expression is self-adjoint and positive semi-definite if $(I + tP^{1/2} S P^{1/2})^{-1}$ is. But if $x \in H$ and $y = (I + tP^{1/2} S P^{1/2})^{-1}x$ then
\begin{align*}
  \iprod{(I + tP^{1/2} S P^{1/2})^{-1}x, x} &= \iprod{y, y + tP^{1/2} S P^{1/2}y} \\
                                         &= \iprod{y, y} + t\iprod{P^{1/2}y, SP^{1/2}y} \ge 0
\end{align*}
as $S$ is positive semi-definite. Thus $\exp{tG}P \in \SHP$.
\end{proof}

\section{Convergence analysis for Lie splitting}\label{sec:analysis_Lie}

We now turn to the convergence analysis and consider first Lie splitting scheme for~\eqref{eq:DREclassical}. We consider an equidistant grid $t_n = n\tau$ with time step $\tau = T/N$ and approximate $P^n \approx P(t_n)$. In the notation of Section~\ref{sec:setting}, the method is then defined by $P^{n+1} = \cLtau P^n$, where the time-stepping operator $\cLtau$ is given by
\begin{align*}
  \cLtau = \exp{\tau F}\exp{\tau G}.
\end{align*}

We start by expressing one step of the method in a more useful way. By expanding the nonlinear flow in a first-order Taylor expansion, we get for any $P\in\SHP$ that
\begin{equation}\label{eq:LieScheme}
  \begin{split}
    \cLtau P &= \exp{\tau \cA} P  + \int_{0}^{\tau} \exp{(\tau - s)\cA} Q \, \ds + \tau \exp{\tau\cA} GP
    + \int_{0}^{\tau} (\tau-s ) \exp{\tau\cA} \ddds \exp{sG}P  \, \ds \\
    &=: L_0(P) +  L_1 + L_2(P) + L_3(P).
  \end{split}
\end{equation}
Next, we similarly reformulate the exact solution. Since Lemma~\ref{lemma:existence} guarantees $P \in \CHk{\infty}{[\eps, T]}$, we can take any $\eps > 0$ and do the following Taylor expansion :
\begin{equation}\label{eq:exactsol21}
  \begin{aligned}
    P(\tau)x &= P(\tau - \eps + \eps ) \\
    &= \exp{(\tau-\eps)\cA}P(\eps)
      + \int_{\eps}^{\tau} \exp{(\tau-s)\cA} \bigl( Q + GP(s) \bigr) \, \ds.
\end{aligned}
\end{equation}
According to Lemma~\ref{lemma:existence}, the operator $\exp{(\tau-s)\cA}$ maps any initial condition into a $\CHk{\infty}{[0, \tau - \eps]}$-function. The composition $s \mapsto \exp{(\tau-s)\cA}GP(s)$ is therefore in $\CHk{\infty}{[\eps, \tau - \eps]}$, since $s \mapsto GP(s) \in \CHk{\infty}{[\eps, T]}$ like $P$ itself.
We can therefore expand $\exp{(\tau-s)\cA}GP(s)$ too, which yields
\begin{equation*}
\begin{split}
    P(\tau) &= \exp{(\tau-\eps)\cA} P(\eps)
        + \int_{\eps}^{\tau} \exp{(\tau-s)\cA} Q  \, \ds
        + (\tau-\eps ) \exp{(\tau-\eps)\cA} GP(\eps) \\
        &\quad+ \int_{\eps}^{\tau-\eps} \int_{\eps}^{s} \ddr \Big(\exp{(\tau-r)\cA} GP(r) \Big) \, \dr \ds
        + \int_{\tau-\eps}^{\tau} \exp{(\tau-s)\cA} GP(s) \, \ds \\
    &=: I_0(\eps,P_0) + I_1(\eps,P_0)  + I_2(\eps,P_0) + I_3(\eps,P_0) + I_4(\eps, P_0).
\end{split}
\end{equation*}
Now let $x \in H$. Since $t\mapsto P(t)x$ is continuous, the first three terms satisfy 
\begin{equation*}
  \lim_{\eps \rightarrow 0^+} I_j(\eps,P_0)x = L_j(P_0)x\quad\text{ for }j = 0, 1, 2.
\end{equation*}
That is, the first three terms agree with the corresponding terms in the expansion of the scheme given in~\eqref{eq:LieScheme}. Further,
\begin{equation*}
  \maxnorm{s \mapsto \exp{(\tau-s)\cA} GP(s)} \le C\maxnorm{P}^2,
\end{equation*}
by Lemma~\ref{lemma:analyticSG}, so that
\begin{equation*}
\lim_{\eps \rightarrow 0^+} \normLHH{I_4(\eps, P_0)} = 0.
\end{equation*}
As a consequence, also $\lim_{\eps \rightarrow 0^+} I_3(\eps, P_0)x$ exists and equals $R(0,x)$, where
\begin{equation*}
R(t,x) = \int_{0}^{\tau} \int_{0}^{s} \ddr \Big(\exp{(\tau-r)\cA} GP(t + r) x\Big) \, \dr \ds
\end{equation*}
for $t\in[0,T-\tau]$. We can do the same expansions around $P(k\tau)$ instead of $P_0$ and thus have that
\begin{equation}\label{eq:Lie_expansion}
P((k+1)\tau)x = (L_0 + L_1 + L_2)(P(k\tau))x + R(k\tau,x),
\end{equation}
Ideally, we would now argue that $\norm{R(t,x) - L_3(P(k\tau))x} \le C\tau^2\norm{x}$, which would give us a bound for the local error of the method. However, this does not necessarily hold under only Assumption~\ref{ass:operators}. Thus, we now introduce the following extra assumptions:
\begin{assumption}\label{ass:P0regular}
  The initial condition $P_0 \in \domain{\cA}$.
\end{assumption}
\begin{assumption}\label{ass:Sregular}
 The operators $S$ and $AS$ both belong to $\cL(H)$.
\end{assumption}
\noindent Either Assumption~\ref{ass:P0regular} or Assumption~\ref{ass:Sregular} will lead to first-order convergence, except for a $|\log \tau|$-factor in the former case. This factor arises from using the analyticity properties of $\exp{tA}$ to essentially ignore one application of $\cA$, which allows us to only assume $P_0 \in \domain{\cA}$ rather than the usual $P_0 \in \domain{\cA^2}$. If the latter does hold, the proof can be modified to remove the $|\log \tau|$-factor.

\subsection{Lie splitting under Assumption~\ref{ass:P0regular}}\label{subsec:Lie_ass2}
Our approach in the case of Assumption~\ref{ass:P0regular} is similar to the one used in~\cite{OstermannPiazzolaWalach2018} for the matrix-valued case. We note first that
since the solution $P$ to~\eqref{eq:DREclassical} is in $\CpHk{1}{[0,T]}$ by Lemma~\ref{lemma:regularity}, there exists a $\gamma>0$ such that 
\begin{equation*}
  \max\{\maxnorm{P},\maxnorm{\dot{P}}\}\leq \gamma. 
\end{equation*}
We start with the following preliminary result:
\begin{lemma}\label{lemma:Lie1prelresult}
Let Assumptions~\ref{ass:operators} and~\ref{ass:P0regular} be fulfilled and let $P$ be the solution to~\eqref{eq:DREclassical}. Then the equality 
\begin{equation*}
R(t,x)=\int_{0}^{\tau} \int_{0}^{s} \lambda_1(t,r)x\, \dr \ds + \int_{0}^{\tau} \int_{0}^{s} \lambda_2(t,r)x  \, \dr \ds
\end{equation*}
of, improper, integrals is well defined for all $t\in [0,T-\tau]$ and $x\in H$. Here,
 \begin{align*}
    \lambda_1(t,r) &= \cA\exp{(\tau-r)\cA}  \big(P(t+r)SP(t+r)\big)\quad\text{and} \\
    \lambda_2(t,r) &=  -\exp{(\tau-r)\cA} \big(\dot{P}(t+r)SP(t+r) + P(t+r)S\dot{P}(t+r)\big).
  \end{align*}
\end{lemma}
\begin{proof}
Lemma~\ref{lemma:existence} and the chain rule gives that 
\begin{equation*}
\ddr \Big(\exp{(\tau-r)\cA} GP(t + r) \Big)=\lambda_1(t,r)+\lambda_2(t,r)
\end{equation*}
for $t+r > 0$ with $0 \le r < \tau$, where all three functions are in $\CHk{\infty}{[\eps, \tau - \eps]}$ with respect to $r$. Since $P \in \CpHk{1}{[0,T]}$, one has the bound
\begin{equation}\label{eq:I2bound}
 \sup_{r\in (0,\tau)}\normLHH{\lambda_2(t,r)} \le C \maxnorm{P} \maxnorm{\dot{P}} \normLHH{S} \le C(\gamma)
\end{equation}
for all $t\in [0,T-\tau]$. That is, $r\mapsto \lambda_2(t,r)$ is uniformly bounded and the integral $\int_{0}^{\tau} \int_{0}^{s} \lambda_2(t,r)x\, \dr \ds$ is therefore well defined. Furthermore,
\begin{equation*}
\lim_{\eps \rightarrow 0^+} \int_{\eps}^{\tau-\eps}\int_{\eps}^{s} \lambda_1(t,r)x\, \dr \ds = R(t,x)-\int_{0}^{\tau} \int_{0}^{s} \lambda_2(t,r)x\, \dr \ds.
\end{equation*}
Hence, the integral $\int_{0}^{\tau} \int_{0}^{s} \lambda_1(t,r)x\, \dr \ds$ is also well defined and that the sought after equality holds. 
\end{proof}

We can now prove consistency for the Lie scheme: 
\begin{lemma}\label{lemma:Lie1Cons}
Let Assumptions~\ref{ass:operators} and~\ref{ass:P0regular} be fulfilled and let $P$ be the solution to~\eqref{eq:DREclassical}.
If $\normLHH{\cLtau^kP_0} \leq \gamma+1$ for all $k=0,\ldots,n$, then 
\begin{equation*}
    \normLHH{P\bigl((n+1)\tau\bigr) - \cLtau^{n+1}P_0} \leq D\tau(1 + \vert \log \tau \vert),
  \end{equation*}
where $D=D(\gamma)$.  
\end{lemma}
\begin{proof}
First observe that the hypothesis implies that the first $n+1$ errors, i.e.\ 
\begin{equation*}
E_k=P(k\tau) - \cLtau^{k}P_0,\quad \text{with }k=0,\ldots,n, 
\end{equation*}
are all elements in $\cL(H)$. By~\eqref{eq:Lie_expansion} and~\eqref{eq:LieScheme} it follows that
\begin{align*}
  E_{n+1}x &= P\bigl((n+1)\tau\bigr)x - \cLtau^{n+1}P_0x \\
            &= L_0\bigl(P(n\tau) - \cLtau^{n} P_0 \bigr)x + \bigl(L_2(P(n\tau)) - L_2(\cLtau^{n}P_0) \bigr)x  \\
       &\quad + R(n\tau,x) - L_3(\cLtau^{n} P_0)x 
\end{align*}
Noting that the first term is $\exp{\tau\cA}E_{n}x$, we can repeatedly apply this equality and obtain the error representation
\begin{equation*}
  \begin{split}
    E_{n+1}x&= L_0(E_0)x + \sum_{k=1}^{n} \exp{(n-k)\tau \cA}  \bigl( L_2(P(k\tau)) - L_2(\cLtau^{k}P_0) \bigr)x \\
    &\quad+ \sum_{k=0}^{n} \exp{(n-k)\tau \cA }\bigl(R(k\tau,x) - L_3(\cLtau^{k} P_0)\bigr)x.
  \end{split}
\end{equation*}
Since $E_0 = 0$, this representation gives the bound 
  \begin{equation}\label{eq:error_recursion_norm}
    \begin{split}
      \norm{ E_{n+1}x}
        &\leq C\sum_{k=1}^{n} \norm[\big]{\bigl( L_2(P(k\tau)) - L_2(\cLtau^k P_0)  \bigr)x} 
        + C\sum_{k=0}^{n} \norm[\big]{L_3(\cLtau^kP_0)x} \\
        &\quad+ \norm[\bigg]{  \sum_{k=0}^{n} \exp{\tau(n-k)\cA } R(k\tau,x)} =: K_1+K_2+K_3.
      \end{split}
  \end{equation}    
As $L_2(P) = \tau\exp{\tau\cA}GP$, the hypothesis and Lemma~\ref{lemma:LipBounds} with $\rho = \gamma+1$ yield the bound
\begin{equation*}
K_1\leq C(\gamma)\, \tau \sum_{k=1}^{n} \normLHH{E_k}\norm{x}.
\end{equation*}
Furthermore, Lemma~\ref{lemma:NlinFlow} with $\rho = \gamma + 1$, the trivial bound $k\tau\leq T$, and the hypothesis give 
\begin{equation*}
K_2\leq C(\gamma)\, \tau \norm{x}.
\end{equation*}
In order to bound $K_3$, we recall Lemma~\ref{lemma:Lie1prelresult} and obtain
\begin{equation*}
K_3 \le K_{31} + K_{32},
\end{equation*}
where
\begin{equation*}
  K_{3j} = \norm[\bigg]{ \sum_{k=0}^{n} \exp{\tau(n-k)\cA } \int_{0}^{\tau} \int_{0}^{s} \lambda_j(k\tau,r)x   \, \dr \ds }, \quad j = 1, 2,
\end{equation*}
with $\lambda_j$ given in Lemma~\ref{lemma:Lie1prelresult}. By~\eqref{eq:I2bound} one has the bound
\begin{equation*}
  K_{32} \le C\sum_{k=0}^{n}\norm[\bigg]{ \int_{0}^{\tau} \int_{0}^{s} \lambda_2(k\tau,r)x \, \dr \ds }\le C(\gamma)\tau \norm{x}.
\end{equation*}
To bound $K_{31}$, we note that
\begin{equation*}
  \norm{\lambda_1(t,r)x} \le \frac{C}{\tau-r} \maxnorm{P}^2\normLHH{S} \norm{x}.
\end{equation*}
Directly using the equality $\int_{0}^{\tau} \int_{0}^{s} \frac{1}{\tau-r} \, \dr \ds = \tau$ on each of the summands does not result in a good enough bound. We only do this for the last summand with $k=n$, i.e.\
  \begin{equation*}
    \norm[\bigg]{\int_{0}^{\tau} \int_{0}^{s} \lambda_1(n\tau,r)x \,\dr \ds} \leq C(\gamma)\tau \norm{x}.
  \end{equation*}
  For the remaining summands, we move the outer exponential into the integrals and interchange it with the one in the integrand by using the relation~$\cA\exp{t\cA} = \exp{t\cA}\cA$. Lemma~\ref{lemma:analyticSG} then gives the bound
    \begin{equation*}
      \begin{split}
        \norm[\big]{\sum_{k=0}^{n-1} &\exp{\tau(n-k)\cA } \int_{0}^{\tau} \int_{0}^{s} \lambda_1(k\tau,r)x \,\dr \ds}\\
        &\leq \sum_{k=0}^ {n-1}  \int_{0}^{\tau} \int_{0}^{s} \norm[\big]{ \exp{(\tau - r)\cA}\cA \exp{\tau(n-k)\cA }  GP(k\tau + r)x} \, \dr \ds \\
    &\leq C \sum_{k=0}^ {n-1}  \int_{0}^{\tau} \int_{0}^{s} \norm[\big]{
       \cA \exp{\tau(n-k)\cA } GP(k\tau + r)x} \, \dr \ds \\
    &\leq C\sum_{k=0}^ {n-1}  \int_{0}^{\tau} \int_{0}^{s}
        \frac{1}{(n-k)\tau} \norm{ GP(k\tau + r)x} \, \dr \ds \\
    &\leq C \maxnorm{GP}\,\tau^2 \sum_{k=0}^{n-1} \frac{1}{(n-k)\tau}\norm{x}\\
    &\leq C(\gamma)\, \tau \bigl(1 + \log n\bigr)\norm{x}.
    \end{split}
  \end{equation*}
In total, 
  \begin{equation*}
  \begin{aligned}
      \normLHH{E_{n+1}}&\leq\sup_{\norm{x}=1}K_1+K_2+K_3\\
      &\leq C(\gamma)\bigl(\tau\sum_{k=0}^{n} \normLHH{E_k} + \tau(1 + \vert \log \tau \vert)\bigr).
    \end{aligned}
\end{equation*}
The sought after error bound then follows by a discrete Grönwall inequality, see e.g.~\cite[Theorem~4.1]{Emmrich1999}.
\end{proof}

\begin{remark}
  We note that the condition $\normLHH{\cLtau^kP_0} \leq \gamma+1$ is an arbitrary choice which could be replaced by $\normLHH{\cLtau^kP_0} \leq \gamma + \rho$ for any $\rho > 0$ in order to tune the error constants.
\end{remark}

\begin{theorem}\label{thm:Lie1Bound}
Let Assumptions~\ref{ass:operators} and~\ref{ass:P0regular} be fulfilled and let $P$ be the solution to~\eqref{eq:DREclassical}.
Then
\begin{equation*}
    \normLHH{P(n\tau) - \cLtau^nP_0} \leq D\tau(1 + \vert \log \tau \vert)\quad\text{for all }\tau \leq \tau^*,
\end{equation*}
where the constants $D, \tau^*$ only depend on the solution $(\maxnorm{P},\maxnorm{\dot{P}})$ and the problem data $(T, \normLHH{S}, \normLHH{Q})$.
\end{theorem}
\begin{proof}
With the notation from Lemma~\ref{lemma:Lie1Cons}, there exists an $\tau^*$ such that
\begin{equation*}
D\tau(1 + \vert \log \tau \vert)\leq 1\quad\text{for all }\tau \leq \tau^*.
\end{equation*}
Let $\tau \leq \tau^*$ and assume that $\normLHH{E_{k}}\leq D\tau(1 + \vert \log \tau \vert)$ for $k=0,\ldots,n$. Then 
\begin{equation*}
\normLHH{\cLtau^kP_0} \leq \normLHH{P(k\tau)}+\normLHH{E_{k}}\leq\gamma+1
\end{equation*}
and  Lemma~\ref{lemma:Lie1Cons} implies that $\normLHH{E_{n+1}}\leq D\tau(1 + \vert \log \tau \vert)$. As $\normLHH{E_{0}}=0$, the first-order convergence follows by induction. 
\end{proof}

\subsection{Lie splitting under Assumption~\ref{ass:Sregular}}

We once more consider the Lie splitting scheme, but with Assumption~\ref{ass:Sregular} rather than Assumption~\ref{ass:P0regular}. In this case, the solution $P$ to~\eqref{eq:DREclassical} guaranteed by Lemma~\ref{lemma:regularity} is only in $\CpH{[0,T]}$ rather than  $\CpHk{1}{[0,T]}$. We therefore set choose a $\gamma>0$ such that
\begin{equation*}
\maxnorm{P}\leq \gamma
\end{equation*}
in this section. We start the analysis by showing that $AS + S A^*$ is bounded and that $G$ is invariant under $\domain{\cA}$.
\begin{lemma}\label{lemma:S}
Let Assumptions~\ref{ass:operators} and~\ref{ass:Sregular} be satisfied. Then $AS + S A^* \in \cL(H)$.
\end{lemma}
\begin{proof}
  Since $AS \in \cL(H)$, so is $(AS)^*$. Furthermore, $S = S^*$ by Assumption~\ref{ass:operators}, and we have $(AS)^*x = S^*A^*x = SA^*x$ for all $x \in \domain{A^*}$.
  But $\domain{A^*}$ is dense in $H$, e.g.\ by \cite[Lemma 1.10.5]{Pazy1983}, and hence $(AS)^* = SA^*$ on $\cL(H)$.
\end{proof}
\begin{remark}
  The conclusion of Lemma~\ref{lemma:S} is what is needed for the analysis, and could be directly assumed instead of the stronger condition given by Assumption~\ref{ass:Sregular}. However, this would be more difficult to verify, and Assumption~\ref{ass:Sregular} more clearly shows that $S$ is expected to be smoothing.
\end{remark}
\begin{lemma}\label{lemma:PSP} 
    Let Assumptions~\ref{ass:operators} and~\ref{ass:Sregular} be satisfied and let $P \in \domain{\cA}$. Then $GP \in \domain{\cA}$.
\end{lemma}

\begin{proof}
The operator $GP$ is in the domain of $\cA$ if $-GP$ is, and this is equivalent to the associated operator $\phi_{PSP}(x,y)$ being continuous with respect to the norm on $H$. By Remark~\ref{rem:basic_operator_properties} it suffices to show that $\phi_{PSP(x,y)}$ is continuous for $x,y \in \domain{A}$. We have
    \begin{align*}
        \phi_{PSP}(x,y) &= (PSP x, Ay) + (Ax, PSPy) \\
        &= (PSP x, Ay) + (Ax, PSPy)  + (ASPx, Py) + (Px, ASPy) \\
                        &\quad - (ASPx, Py) + (Px, ASPy),
    \end{align*}
    since $AS$ is an element of $\cL(H)$ by Assumption~\ref{ass:Sregular}. Now if we let $x' = SPx$ and $y' = SPy$ we find
    \begin{align*}
      \phi_{PSP}(x,y) &= (Px', Ay) + (Ax, Py') + (Px, Ay') + (Ax', Py) \\
                      &\quad- ((AS + SA^*)Px, Py) \\
                      &= \phi_P(x', y) + \phi_P(x,y') - ((AS + SA^*)Px, Py).
    \end{align*}
    Since $P \in \domain{\cA}$ we know that $\vert \phi_P(x,y) \vert \leq C \norm{x} \norm{y}$, which then implies that
    \begin{align*}
      \vert \phi_{PSP}(x,y) \vert  &\leq \vert \phi_P(x',y) \vert  + \vert \phi_P(x,y') \vert  
        + \normLHH{AS + SA^*}\normLHH{P}^2 \norm{x}\norm{y} \\
      &\leq C(\norm{x'}\norm{y} + \norm{x}\norm{y'}) + C\norm{x}\norm{y} \\
      &\leq 2 C \normLHH{P} \normLHH{S} \norm{x}\norm{y} + C\norm{x}\norm{y}.
    \end{align*}
    Thus $\phi_{PSP}$ is continuous with respect to the norm of $H$, so $GP \in \domain{\cA}$.
\end{proof}

Bounding the rest term $R(k\tau,x)$ is the key step in the proof of Lemma~\ref{lemma:Lie1Cons}. The argument used there relies on $\dot{P}(t)x$ being uniformly bounded, which does not necessarily hold under Assumption~\ref{ass:Sregular} instead of~\ref{ass:P0regular}. The next lemma shows that under Assumption~\ref{ass:Sregular} we can express the integrand of $R(k\tau,x)$ in terms of only $P(t)x$ and thereby avoid this issue.

\begin{lemma}\label{lemma:PSPderivative}
  Let Assumptions~\ref{ass:operators} and~\ref{ass:Sregular} be satisfied and let $P\in\CpH{[0,T]}$~be the solution to~\eqref{eq:DREclassical}. Then for $r \in [0, \tau)$ and $t+r \in (0,T]$ the identification 
  \begin{equation*}
    \ddr \bigl(\exp{(\tau-r)\cA} GP(t+r) \bigr) = \exp{(\tau-r)\cA} f(t + r),
  \end{equation*}
  holds, where $f : [0,T] \to \cL(H)$ is defined by
  \begin{equation}\label{eq:f}
    f(s) = -QSP(s) - P(s)SQ + 2P(s) S P(s) S P(s) - P(s)(AS + S A^*)P(s).
  \end{equation}
  As a consequence,
  \begin{equation}\label{eq:limit_PSP_integral}
    \norm[\big]{R(t,x)}_{\cL(H)} \le C\tau^2\norm{x}
  \end{equation}
  where $C = C(\gamma, \normLHH{AS})$.
\end{lemma}

\begin{proof}
  By Lemma~\ref{lemma:existence}, the solution $P$ is an element in $\CHk{\infty}{[\eps, T]}$ and $P(t)\in\domain{\cA}$ for all $t\in [\eps, T]$,  with $\eps>0$. Thus, like in the proof of Lemma~\ref{lemma:Lie1prelresult}, for $t+r > 0$ with $0 \le r < \tau$ we have
  \begin{equation*}
    \ddr \bigl(\exp{(\tau-r)\cA} GP(t+r) \bigr) = \lambda_1(t,r) + \lambda_2(t,r),
  \end{equation*}
  where $\lambda_j$, $j=1,2$, are stated in Lemma~\ref{lemma:Lie1prelresult}. In the following, $P$ and $\dot{P}$ are always evaluated at $t+r$, but we do not always write this out for readability reasons. For the same reason, we drop the arguments for the terms $\lambda_j(t,r)$. 
  
By using the differential equation~\eqref{eq:DREclassical} to expand $\dot{P}(t+r)$ and writing one of the operator semigroups in terms of $\exp{\tau A}$ and $\exp{\tau A^*}$, we get
  \begin{align*}
    \lambda_2 
    &= -\exp{(\tau-r)A^*} \bigl((\cA P) S P    + P S(\cA P)\bigr)\exp{(\tau-r)A} \\
    &\quad- \exp{(\tau-r)\cA} \bigl(QSP + PSQ - 2P S P S P\bigr) 
    = \lambda_{21} + \lambda_{22}.
  \end{align*}
  Due to Lemma~\ref{lemma:PSP}, we can expand the operator $\cA$ in $\lambda_{21}$ in terms of $A$ and $A^*$ and then reassemble it again with a different argument, as follows. First let $x \in \domain{A}$. Then also $\exp{(\tau-r)A}x \in \domain{A}$ and
  \begin{align*}
    \lambda_{21}x &= -\exp{(\tau-r)A^*} \bigl( (A^*P + PA) S P    + P S (A^*P + PA)\bigr)\exp{(\tau-r)A}x \\
           &= -\exp{(\tau-r)A^*} ( A^*PSP + PSPA + PASP + PSA^*P )\exp{(\tau-r)A}x \\
           &= -\exp{(\tau-r)A^*} \bigl( \cA (PSP) + P(AS + SA^*)P \bigl)\exp{(\tau-r)A}x \\
           &= -\cA \exp{(\tau-r)\cA} PSP x - \exp{(\tau-r)\cA}P(AS + SA^*)Px.
  \end{align*}
  In the final equality, we have used the commutativity of $\cA$ and $\exp{t\cA}$. Since the operator $\cA \exp{(\tau-r)\cA} PSP - \exp{(\tau-r)\cA}P(AS + SA^*)P$ is well defined in $\cL(H)$ by Lemmas~\ref{lemma:PSP} and~\ref{lemma:S}, and $\domain{A}$ is dense in $H$ (Remark~\ref{rem:basic_operator_properties}), the above sequence of equalities hold for all $x \in H$. We thus have
  \begin{equation*}
    \lambda_{21} = \cA \exp{(\tau-r)\cA} PSP - \exp{(\tau-r)\cA}P(AS + SA^*)P
  \end{equation*}
  in the sense of $\cL(H)$-operators. We now observe that the first term in $\lambda_{21}$ is in fact $\lambda_1$, but with opposite sign. Thus,
  \begin{equation*}
    \ddr \bigl(\exp{(\tau-r)\cA} GP(t+r)\bigr) = \lambda_{22} - \exp{(\tau-r)\cA}P(t+r)(AS + SA^*)P(t+r),
  \end{equation*}
which is precisely $\exp{(\tau-r)\cA} f(t+r)$. This proves the first claim.

Since $\normLHH{AS + SA^*} \le C$ by Lemma~\ref{lemma:S}, we then directly get by Lemma~\ref{lemma:analyticSG} that the integrand in~\eqref{eq:limit_PSP_integral} satisfies
  \begin{equation*}
      \normLHH{\exp{(\tau-r)\cA}f(t+r)} \le C(\gamma, \normLHH{AS}).
      \end{equation*}
  Since this bound is independent of $r$, and $\tau \leq T$,
  \begin{align*}
    \norm{R(t,x)} &= \norm[\bigg]{\int_{0}^{\tau} \int_{0}^{s} \ddr \Big(\exp{(\tau-r)\cA} GP(t + r) x\Big) \, \dr \ds} \\
                  &\le \int_{0}^{\tau} \int_{0}^{s} \norm{\exp{(\tau-r)\cA} f(t+r)x} \, \dr \ds \\
    &\le C(\gamma, \normLHH{AS})\tau^2\norm{x},
  \end{align*}
  which proves the second claim.
\end{proof}

\begin{lemma}\label{lemma:Lie2Cons}
Let Assumptions~\ref{ass:operators} and~\ref{ass:Sregular} be fulfilled and let $P$ be the solution to~\eqref{eq:DREclassical}.
If $\normLHH{\cLtau^kP_0} \leq \gamma + 1$ for all $k=0,\ldots,n$, then 
\begin{equation*}
    \normLHH{P\bigl((n+1)\tau \bigr) - \cLtau^{n+1}P_0} \leq D\tau,
  \end{equation*}
where $D=D(\gamma, \normLHH{AS})$.  
\end{lemma}
\begin{proof}
  By following exactly the same approach as in the proof of Lemma~\ref{lemma:Lie1Cons}, we find that
  \begin{equation}
    \norm{E_{n+1}x} = K_1+K_2+K_3,
  \end{equation}
  where $K_j$ is defined in~\eqref{eq:error_recursion_norm} and where
  \begin{equation*}
    K_1+K_2\leq C(\gamma)\, \tau \bigl(\sum_{k=1}^{n} \normLHH{E_k}+1\bigr)\norm{x}.
  \end{equation*}
Furthermore, by Lemmas~\ref{lemma:analyticSG} and~\ref{lemma:PSPderivative} and the trivial bound $k\tau\leq T$, we obtain 
\begin{equation*}
  K_3 = \norm[\bigg]{  \sum_{k=0}^{n} \exp{\tau(n-k)\cA} R(k\tau,x)} \leq C(\gamma, \normLHH{AS})\, \tau \norm{x}.
\end{equation*}
Summing up yields
  \begin{equation*}
      \normLHH{E_{n+1}}\leq\sup_{\norm{x}=1}K_1+K_2+K_3\leq C\tau \bigl(\sum_{k=0}^{n} \normLHH{E_k} +1\bigr),
  \end{equation*}
where $C=C(\gamma, \normLHH{AS})$, and the sought after error bound again follows by a discrete Grönwall inequality~\cite[Theorem~4.1]{Emmrich1999}.
\end{proof}

Lemma~\ref{lemma:Lie2Cons} together with the proof of Theorem~\ref{thm:Lie1Bound} then give us our second convergence result:
\begin{theorem}\label{thm:Lie2Bound}
Let Assumptions~\ref{ass:operators} and~\ref{ass:Sregular}  be fulfilled and let $P$ be the solution to~\eqref{eq:DREclassical}.
Then
\begin{equation*}
    \normLHH{P(n\tau) - \cLtau^nP_0} \leq D\tau \quad\text{for all }\tau \leq \tau^*,
\end{equation*}
where the constants $D, \tau^*$ only depend on the solution $(\maxnorm{P})$ and the problem data $(T,\normLHH{S},  \normLHH{Q},\normLHH{AS})$.   
\end{theorem}
 \begin{remark}
 We note that unlike in the proof of Theorem~\ref{thm:Lie1Bound} we did not use the bound on $\cA\exp{t\cA}$. This was possible because the terms can be rearranged such that all direct applications of $A$ occur in $AS + SA^*$, which is bounded under Assumption~\ref{ass:Sregular}.
\end{remark}

\section{Convergence analysis for Strang splitting}\label{sec:analysis_Strang}

In this section, we analyze the Strang splitting scheme for~\eqref{eq:DREclassical} under the Assumptions~\ref{ass:operators} to~\ref{ass:Sregular}. We consider an equidistant grid $t_n = n\tau$ with time step $\tau = T/N$ and approximate $P^n \approx P(t_n)$. In the notation of Section~\ref{sec:setting}, the method is then defined by $P^{n+1} = \cStau P^n$, where the time-stepping operator $\cStau$ is given by
\begin{align*}
  \cStau = \exp{\tau/2 F}\exp{\tau G}\exp{\tau/2 F}.
\end{align*}
A second-order Taylor expansion of the nonlinear flow in $\cStau$ gives the representation
\begin{equation}\label{eq:StrangScheme}
\begin{aligned}
\cStau  P &= 
  \exp{\tau\cA} P + \int_0^{\tau} \exp{(\tau-s)\cA} Q \, \ds + \tau \exp{\tau/2\cA} (G + \tau M)\exp{\tau/2F}P \\
         &\quad+ \exp{\tau/2\cA} \int_0^{\tau} \frac{(\tau-s)^2}{2}  \dddds \exp{sG} \exp{\tau/2 F} P \, \ds \\
    &=: S_0(P) + S_1 + S_2(P) +S_3(P)
\end{aligned}
\end{equation}
in $\cL(H)$, where $M$ is defined in Lemma~\ref{lemma:LipBounds}. Next, we reformulate the exact solution. As for the reformulation in the Lie case, we make use of the fact that the mapping $s \mapsto \exp{(\tau-s)\cA}GP(s)$ is in $\CHk{\infty}{[\eps, \tau - \eps]}$ and obtain the reformulation
\begin{equation*}
  \begin{aligned}
    P(\tau) &= P(\tau - \eps + \eps ) \\
    &= \exp{(\tau-\eps)\cA}P(\eps)
      + \int_{\eps}^{\tau} \exp{(\tau-s)\cA} \bigl( Q + GP(s) \bigr) \, \ds \\
    &= \exp{(\tau-\eps)\cA} P(\eps)
      + \int_{\eps}^{\tau} \exp{(\tau-s)\cA} Q  \, \ds  + \int_{\tau - \eps}^{\tau} \exp{(\tau-s)\cA} GP(s) \, \ds\\
      &\quad + \exp{\tau/2 \cA} GP(\tau/2) \int_{\eps}^{\tau-\eps} 1 \,\ds+ \ddr \Bigl ( \exp{(\tau-r)\cA} GP(r) \Bigr )\Bigl \vert_{r=\tau/2} \int_{\eps}^{\tau-\eps} (s-\tau/2)  \,\ds \\
      &\quad + \int_{\eps}^{\tau-\eps} \int_{\tau/2}^s  (s - r) \dddr \Bigl ( \exp{(\tau-r)\cA} GP(r) \Bigr ) \, \dr \ds \\
    & = I_0(\eps, P_0) + I_1(\eps) + I_2(\eps, P_0) + I_3(\eps, P_0)+ I_4(\eps, P_0)+ I_5(\eps, P_0).
  \end{aligned}
\end{equation*}
From this we find that for any $x \in H$ the first two terms match the expansion of the method, i.e.\
\begin{equation*}
\lim_{\eps \rightarrow 0} I_0(\eps, P_0) x = S_0(P_0)x\quad\text{and}\quad
\lim_{\eps \rightarrow 0} I_1(\eps) x = S_1 x.
\end{equation*}
The integrand of $I_2$  is uniformly bounded on $[\tau - \eps, \tau]$ and the integrand of $I_4$ is symmetric, which imply that
\begin{equation*}
  \lim_{\eps \rightarrow 0} I_j(\eps, P_0)x= 0 \quad \text{for } j = 2,4.
\end{equation*}
Finally, we note that 
\begin{equation*}
  I_3(P_0) = \lim_{\eps \rightarrow 0} I_3(\eps, P_0) = \tau \exp{\tau/2 \cA} GP(\tau/2).
\end{equation*}
Thus we conclude that $\lim_{\eps \rightarrow 0} I_5(\eps, P_0)x $ exists and equals $R(0,x)$ where
\begin{equation*}
  R(t,x) = \int_{0}^{\tau} \int_{\tau/2}^s (s - r) \dddr \Bigl ( \exp{(\tau-s)\cA} GP(t + r) \Bigr ) x \, \dr \ds.
\end{equation*}
This reformulation of the exact solution can be made not only at $t = 0$, but around any $t = k\tau \leq T - \tau$ resulting in
\begin{equation*}
  P((k+1)\tau)x = S_0(P(k\tau))x + S_1x +  I_3(P(k\tau))x + R(k\tau, x).
\end{equation*}
Throughout this section we once more choose a $\gamma>0$ such that 
\begin{equation*}
 \max\{\maxnorm{P}, \maxnorm{\dot{P}}\}\leq\gamma.
\end{equation*}
Before we prove consistency for the Strang splitting, we derive two preliminary results: 
\begin{lemma}\label{lemma:SgSBound}
  Let Assumptions~\ref{ass:operators} and~\ref{ass:Sregular} be fulfilled, then
 \begin{equation*}
    \normLHH{\exp{\tau A} S \exp{\tau A^*} - S} \leq C\tau, 
 \end{equation*}
  where $C = C(\normLHH{AS})$.
\end{lemma}
\begin{proof}
  By expressing the difference as an integral,
  \begin{equation*}
  \begin{split}
      \norm{\bigl ( \exp{\tau A} S \exp{\tau A^*} - S \bigr ) x} \quad &=\norm[\bigg]{ \int_0^{\tau} A \exp{s A} S \exp{s A^*}x + \exp{s A} S A^* \exp{s A^*} x \, \ds}  \\
          &=\norm[\bigg]{ \int_0^{\tau} \exp{s A} AS \exp{s A^*}x + \exp{s A} S A^* \exp{s A^*} x \, \ds}\\
          &\leq  C \big ( \normLHH{AS} + \normLHH{SA^*}\big ) \tau \norm{x}.
  \end{split}
  \end{equation*}
  Note that commuting $A$ and $\exp{sA}$ is justified, since $AS$ is a bounded operator.
\end{proof}
\begin{lemma} \label{lemma:PSPderivative2}
Let Assumptions~\ref{ass:operators} to~\ref{ass:Sregular} be satisfied and let $P$ be the solution to~\eqref{eq:DREclassical}. 
Then, with $f$ as in Lemma~\ref{lemma:PSPderivative}, the equality 
\begin{equation*}
    \dddr \bigl(\exp{(\tau-r)\cA} GP(t+r)\bigr) = -\cA\exp{(\tau-r)\cA} f(t+r) + \exp{(\tau-r)\cA} \ddr f(t+r),
\end{equation*}
is valid for $r \in [0, \tau)$ and $t+r \in (0, T]$. Furthermore, 
\begin{equation*}
  \norm[\bigg]{\int_{0}^{\tau} \int_{\tau/2}^s ( s - r)  \exp{(\tau-r)\cA}\ddr f(t+r) x \, \dr \ds } \le C\tau^3\norm{x}
\end{equation*}
with a constant $C=C(\gamma, \normLHH{AS})$.
\end{lemma}
\begin{proof}
As $0 \le r < \tau$, $0<t+r \leq T$ and Assumption~\ref{ass:Sregular} is fulfilled, Lemma~\ref{lemma:PSPderivative} gives 
  \begin{equation*}
    \ddr \Big(\exp{(\tau-r)\cA} GP(t+r)\Big) = \exp{(\tau-r)\cA} f(t+r).
  \end{equation*}
Since $f$ is an element of $\CHk{\infty}{[\eps, T]}$, as it is a product of elements in $\cL(H)$ and the solution $P\in\CHk{\infty}{[\eps, T]}$, we obtain 
  \begin{equation*}
    \ddr \exp{(\tau-r)\cA} f(r) = -\cA\exp{(\tau-r)\cA} f(r) + \exp{(\tau-r) \cA} \ddr f(r),
  \end{equation*}
i.e.\ the desired equality. Writing out the last term yields the identity 
  \begin{equation*}
    \begin{aligned}
      \ddr f(t+r)=&-QS\dot{P} - \dot{P}SQ + 2 (\dot{P}SPSP + PS\dot{P}SP + PSPS\dot{P} ) \\
                &- \dot{P} (AS + SA^*) P - P(AS + SA^*) \dot{P},
  \end{aligned}
  \end{equation*}
and the bound 
 \begin{equation}\label{eq:ftermbound}
  \sup_{r\in(0,\tau)}\norm[\Big]{ \exp{(\tau-r)\cA} \ddr f(t+r)}_{\cL(H)}\leq C\maxnorm[\Big]{ \ddr f}\leq C(\gamma, \normLHH{AS}).
 \end{equation}
 Hence, 
 \begin{equation*}
    \begin{aligned}
    \norm[\bigg]{\int_{0}^{\tau} 
    &\int_{\tau/2}^s (s - r)  \exp{(\tau-r)\cA}\ddr f(r) x \, \dr \ds} \\ 
    &\leq  \int_{0}^{\tau/2}\int_s^{\tau/2} (s - r)\norm[\Big]{\exp{(\tau-r)\cA} \ddr f(t+r)x} \, \dr \ds \\
    &\quad + \int_{\tau/2}^{\tau}\int_{\tau/2}^s (s - r) \norm[\Big]{\exp{(\tau-r)\cA} \ddr f(t+r)x }\, \dr \ds \\
    &\leq C \maxnorm[\Big]{\ddr f} \int_{0}^{\tau/2}\int_s^{\tau/2} (s - r) \, \dr \ds \norm{x}\\ 
    & \leq C(\gamma,\normLHH{AS}) \tau^3 \norm{x},
    \end{aligned}
 \end{equation*}
  and the desired inequality is shown. We note that the integral was split in two since the direction of integration in the inner integral changes when $s < \tau/2$.
\end{proof}
\begin{remark}
Note that the expression for the second derivative consists of two terms, and we only bound the second term. The first term will be handled separately in the proof of Lemma~\ref{lemma:StrangCons}.
\end{remark}
We can now prove consistency for the Strang splitting:
\begin{lemma}\label{lemma:StrangCons}
 Let Assumptions~\ref{ass:operators} to \ref{ass:Sregular} be fulfilled and let $P$ be the solution to~\eqref{eq:DREclassical}. If $\normLHH{\cStau^k P_0} \leq \gamma + 1$ for all $k = 0, \dots n$, then
 \begin{equation*}
    \normLHH{E_{n+1}} \leq D\tau^2(1 + \vert \log \tau \vert),
 \end{equation*}
  where $D = D(\gamma, \normLHH{AS})$.
\end{lemma}
\begin{proof} 
 Comparing the approximation and exact solution we get
  \begin{equation*}
    \begin{aligned}
      E_{n+1}x 
      &= P((n+1)\tau)x - \cStau^{n+1} P_0 x \\
      &= \exp{\tau \cA} \bigl ( P(n\tau)  -  \cStau^n P_0 \bigr ) x\\
      &\quad+ \bigl ( S_2 (P(n\tau)) - S_2(\cStau^n P_0) \bigr )x \\
      &\quad + \bigl (I_3- S_2\bigr)(P(n\tau))x \\
      &\quad + R(n\tau,x) - S_3(\cStau^n(P_0))x.
    \end{aligned}
  \end{equation*}
Again we see that the first term is $\exp{\tau \cA} E_n x$  so this term can repeatedly be rewritten using the above equality, resulting in
the bound
 \begin{equation*}
    \begin{aligned}
    \norm{E_{n+1}x}
    &\leq C \sum_{k=1}^{n} \norm{ \bigl( S_2(P(k\tau)) - S_2(\cStau^{k}P_0) \bigr)x} \\
    &\quad + C \sum_{k=0}^{n} \norm{(I_3-S_2)(P(k\tau))x} \\
    &\quad + \norm[\bigg]{\sum_{k=0}^{n} \exp{(n-k)\tau \cA } R(k\tau,x) }+ C\sum_{k=0}^{n} \norm{S_3(\cStau^{k} P_0)x} \\
    & =: K_1 + K_2 + K_3 + K_4.
 \end{aligned}
  \end{equation*}
Applying Lemmas~\ref{lemma:analyticSG} and~\ref{lemma:LipBounds} yields
\begin{equation*}
K_1 \leq C(\gamma) \tau(1+ \tau)\sum_{k=1}^n\normLHH{E_n}\norm{x}.
\end{equation*}
The term $K_4$ is similarly straightforward to bound, since the hypothesis implies that 
\begin{equation*}
\normLHH{\exp{\tau/2 F} \cStau^k P_0  x} \leq C(\gamma)\quad\text{for }k = 0, \dots, n,
\end{equation*}
and Lemma~\ref{lemma:NlinFlow} then gives the bound
\begin{equation*}
    K_4 \leq C\sum_{k=0}^n  \maxnorm[\Big]{s \mapsto \dddds \exp{sG} \exp{\tau/2 F}\cStau^k P_0}\int_0^{\tau} \frac{(\tau-s)^2}{2}\,\ds\norm{x} 
           \leq  C(\gamma) \tau^2\norm{x}.
  \end{equation*}
Next, we bound the $K_2$ term, starting by only considering a single summand. With the shorthand 
\begin{equation*}
\hat{P}=\exp{\tau/2 F}P(k\tau)
\end{equation*}
we obtain
\begin{equation}\label{eq:strangdiff}
    \norm{(I_3-S_2)(P(k\tau))x}\leq C \tau \norm{ \bigl(GP((k+1/2)\tau)-G\hat{P}\bigr)x + \tau M\hat{P}x}.
\end{equation}
Introduce the function $g$ defined by
\begin{equation*}
    g(t,s) = \int_0^s \exp{(\tau/2-r)\cA} GP(t + r) \, \dr,
  \end{equation*}
  where $r \mapsto g(t,r) \in \CHk{\infty}{[\eps, \tau/2-\eps]}$ for any $\eps > 0$, with its first two derivatives being uniformly bounded by Lemma~\ref{lemma:PSPderivative}. This gives the representation $P((k+1/2)\tau) =\hat{P} + g(k\tau, \tau/2) $ and the equality
\begin{equation*}
  \begin{aligned}
&GP((k+1/2)\tau)x - G\hat{P}x =  \int_0^{\tau/2} \dds G(\hat{P} + g(k\tau, s) )x  \, \ds  \\
    &\quad= -\frac{\tau}{2} \bigl( g_s(k\tau, 0) S  \hat{P} + \hat{P} S g_s(k\tau,0) \bigr)x
    + \int_0^{\tau/2} \int_0^s \dddr G\bigl(\hat{P} + g(k\tau, r) \bigr)x \, \dr \ds,
  \end{aligned}
\end{equation*}
where $g_s=\partial g/\partial s$. Inserting this equality into~\eqref{eq:strangdiff} and summing up yields
\begin{equation*}
  \begin{aligned}
    K_2&\leq C \sum_{k=0}^{n} \tau^2\norm[\Big]{\bigl( g_s(k\tau,0) S  \hat{P} + \hat{P} S g_s(k\tau,0) - 2M\hat{P}\bigr) x} \\
    &\quad +  C \sum_{k=0}^{n} \tau \,\norm[\bigg]{\int_0^{\tau/2} \int_0^s \dddr G\big(\hat{P} + g(k\tau,r) \bigl) x \, \dr \ds}=:K_{21}+K_{22}.
  \end{aligned}
\end{equation*}
Now let $P$ be shorthand for $P(k\tau)$. For the individual $K_{21}$-summands, writing out $\hat{P} = \exp{\tau/2\cA}P + \int_0^{\tau/2} \exp{(\tau/2-s)\cA}Q \, \ds$ results in many terms, most of which contain the integral as a factor. Since all the other factors in these terms are uniformly bounded, they are all bounded by $C(\gamma)\tau$. We thus get
\begin{equation*}
  \begin{aligned}
   &\norm[\Big]{\bigl( g_s(k\tau,0) S  \hat{P}+ \hat{P} S g_s(k\tau,0) - 2M\hat{P}\bigr)x} \\
  &\quad= \norm{ \bigl ( \exp{\tau/2 \cA} (PSP) S \hat{P}x + \hat{P} S \exp{\tau/2 \cA} (PSP) - 2  \hat{P} S \hat{P} S \hat{P} \bigr ) x}\\
        &\quad\leq  \norm{\bigl ( G\big(\exp{\tau/2 \cA} P \big) S \exp{\tau/2 \cA} P
          - \exp{\tau/2 \cA} (GP) S \exp{\tau/2 \cA}P \bigr ) x} \\
        &\quad\quad + \norm{ \Bigl ( \bigl (\exp{\tau/2 \cA} P \bigr ) S G\big(\exp{\tau/2 \cA} P\big) 
        - \bigl (\exp{\tau/2 \cA} P\bigr) S \exp{\tau/2 \cA} (G P) \Bigr )x} + C(\gamma) \tau \\
        &\quad\le \tau C(\gamma, \normLHH{AS})\norm{x},
  \end{aligned}
\end{equation*}
where the last bound follows from Lemma~\ref{lemma:SgSBound}. 

Furthermore, the integrand of each $K_{22}$-summand can be bounded by 
  \begin{equation*}
      \begin{split}
        \norm[\Big]{&\dddr G\big(\hat{P} + g(k\tau,r) \bigl)}_{\cL(H)} \\
          &= \norm[\Big]{\bigl (\hat{P} + g(k\tau,r) \bigr) S \exp{(\tau/2-r)\cA}f(k\tau + r)
            + \exp{(\tau/2-r)\cA}f(k\tau + r) S \bigl ( \hat{P} + g(k\tau,r) \bigr ) \\
          &\qquad + 2 \exp{(\tau/2-r)\cA} \big(PSP\big) S \exp{(\tau/2-r)\cA} \big(PSP\big)}_{\cL(H)} \\
          &\leq 2C \normLHH{S}\bigl(\maxnorm{f} \sup_{r\in(0,\tau/2)}\norm{g(k\tau,r)} + \maxnorm{f}\normLHH{\hat{P}}+ \maxnorm{P}^4 \normLHH{S}^2\bigr),
      \end{split}
  \end{equation*}
since $ \frac{\partial^2g(t,r)}{\partial r^2} = \exp{(\tau/2-r)\cA}f(t + r)$. Applying the derived bounds on the summands of $K_{2j}$ gives
\begin{equation*}
   K_2\leq C(\gamma,\normLHH{AS}) \tau^2 \norm{x}.
\end{equation*}

Finally we bound $K_3$. By Lemma~\ref{lemma:PSPderivative2} we obtain that
\begin{equation*}
   K_3\leq K_{31}+K_{32}
\end{equation*}
where  
\begin{equation*}
    K_{3j} = \norm[\bigg]{\sum_{k=0}^{n} \exp{(n-k)\tau \cA }\int_{0}^{\tau} \int_{\tau/2}^s \lambda_j(k\tau,r) x \, \dr \ds}\quad\text{for} \quad j = 1,2,
\end{equation*}
with
\begin{equation*}
    \begin{aligned}
      \lambda_1(t,r) &= -(s - r)  \cA \exp{(\tau-r)\cA}f(t+r) \\
      \lambda_2(t,r) &= (s - r)  \exp{(\tau-r)\cA}\ddr f(t+r).
    \end{aligned}
\end{equation*}
Note that the improper integral $\int_0^{\tau} \int_{\tau/2}^s \lambda_1(k\tau, r)x \, \dr \ds$ is well defined since~$r\mapsto\lambda_2(t,r)$ is uniformly bounded by~\eqref{eq:ftermbound} and 
 \begin{equation*}
\lim_{\eps \rightarrow 0^+} \int_{\eps}^{\tau-\eps}\int_{\tau/2}^{s} \lambda_1(k\tau,r) x \, \dr \ds = R(t,x)-\int_{0}^{\tau} \int_{\tau/2}^s \lambda_2(k\tau,r) x \, \dr \ds.
\end{equation*}
From the bound in Lemma~\ref{lemma:PSPderivative2} we directly have
 \begin{equation*}
K_{32}\leq C\sum_{k=0}^{n} \norm[\bigg]{\int_{0}^{\tau} \int_{\tau/2}^s \lambda_2(k\tau,r) x \, \dr \ds}\leq C(\gamma, \normLHH{AS})\tau^2.
 \end{equation*}
For $K_{31}$ we proceed as in Section~\ref{subsec:Lie_ass2}, bounding the $k=n$ term separately:
\begin{equation*}
    \begin{aligned}
      \norm[\Big] {\int_{0}^{\tau} 
      &\int_{\tau/2}^s \lambda_1(n\tau,r) x \, \dr \ds} \leq \int_{0}^{\tau/2} \int_s^{\tau/2} \norm{\lambda_1(n\tau,r) x} \, \dr \ds \\
      &\quad + \int_{\tau/2}^{\tau} \int_{\tau/2}^s \norm{\lambda_1(n\tau,r) x} \, \dr \ds \\
      &\leq C \maxnorm{f} \Bigl ( \int_{0}^{\tau/2}  \int_s^{\tau/2} \frac{r-s}{\tau-r} \, \dr \ds 
        + \int_{\tau/2}^{\tau} \int_{\tau/2}^s \frac{s - r}{\tau-r} \, \dr \ds \Bigr ) \norm{x} \\
      &\leq C(\gamma, \normLHH{AS}) \int_{\tau/2}^{\tau} \int_{\tau/2}^s \frac{s-r}{\tau-r} \, \dr \ds \norm{x}= C(\gamma, \normLHH{AS}) \tau^2\norm{x}.
      \end{aligned}
\end{equation*}
For the rest of the sum we have the bound 
\begin{equation*}
    \begin{aligned}
    \norm[\Big]{\sum_{k=0}^{n-1}
    & \exp{(n-k)\tau  \cA} \int_0^{\tau} \int_{\tau/2}^s \lambda_1(k\tau,r) x \, \dr \ds} \\
    & \leq  \sum_{k=0}^{n-1} \int_0^{\tau/2} \int_s^{\tau/2} \norm{(s - r) \exp{(\tau-r)\cA}   \cA \exp{(n-k)\tau \cA}f(t+r) x} \, \dr \ds \\
    &\quad + \sum_{k=0}^{n-1}\int_{\tau/2}^{\tau} \int_{\tau/2}^s \norm{(s - r) \exp{(\tau-r)\cA} \cA \exp{(n-k)\tau \cA}f(t+r) x }\, \dr \ds\\
    &\leq C \maxnorm{f} \Big ( \sum_{k=0}^{n-1} \int_{\tau/2}^{\tau} \int_{\tau/2}^s \frac{s - r}{(n-k)\tau} \, \dr \ds \Big ) \norm{x} \\
    &\leq C(\gamma, \normLHH{AS}) \tau^3 \sum_{k=0}^{n-1} \frac{1}{(n-k)\tau} \norm{x}\\
    & \leq C(\gamma, \normLHH{AS}) \tau^2 (1 + \log n ) \norm{x}.
\end{aligned}
\end{equation*}
Combining the bounds for $K_j$ give \begin{equation*}
    \begin{aligned}
      \normLHH{E_{n+1}} &= \sup_{\norm{x} = 1} K_1 + K_2 + K_3 + K_4  \\
                        &\leq C(\gamma) \tau \sum_{k=1}^n \normLHH{E_k} + C(\gamma,\normLHH{AS}) \tau^2 (1 + \vert \log \tau \vert).
    \end{aligned}
  \end{equation*}
From this the desired bound is found by a discrete Grönwall inequality.
\end{proof}
The convergence analysis of the Strang splitting scheme now follows by an identical argument to the one in the proof of Theorem~\ref{thm:Lie1Bound}, using the bound provided by
Lemma~\ref{lemma:StrangCons} instead of Lemma~\ref{lemma:Lie1Cons}:
\begin{theorem}\label{thm:StrangBound}
  Let Assumptions~\ref{ass:operators} to \ref{ass:Sregular} be fulfilled and let $P$ be the solution to~\eqref{eq:DREclassical}. Then
  \begin{equation*}
    \normLHH{P(n\tau) - \cStau^n(P_0)}  \leq D\tau^2(1 + \vert \log \tau \vert)\quad \text{for all } \tau \leq \tau^*,
  \end{equation*}
  where the constants $D, \tau^*$ depend only on the solution $(\maxnorm{P}, \maxnorm{\dot{P}})$ and problem data $(T,\normLHH{S},  \normLHH{Q},\normLHH{AS})$.
\end{theorem} 
 
\section{Numerical experiments}\label{sec:experiments}
In order to illustrate the theoretical results, we now present several numerical experiments. They are all based on the connection between DREs and the optimal control of LQR problems mentioned in the introduction. We consider the problem of minimizing the cost functional 
\begin{equation}
  \label{eq:controlledHeat_J}
  J(u) = \int_{0}^{T}{\|y(t)\|^2 + \|u(t)\|^2 \, \dt} + \norm{Gx(T)}^2,
  \end{equation}
  subject to
  \begin{equation}  \label{eq:controlledHeat_system}
    \left\{
   \begin{aligned}
      \dot{x} &= A x + Bu  && \text{in } \Omega \\
      y &= Ex,
    \end{aligned}
    \right.
 \end{equation}
 where $\Omega = (0, 1)$. We set  $H = L^2(\Omega)$ and choose $A = \Delta$, the Laplacian with periodic boundary conditions, and $Ex = \int_{\Omega} x$, such that the problem is a controlled heat equation on the unit interval where we observe the mean of the state $x$.
 The control operator $B: \mathbb{R} \rightarrow H$ and the initial condition $P_0 = G^*G: H \rightarrow H$ will be chosen such that different combinations of Assumptions~\ref{ass:P0regular} and~\ref{ass:Sregular} are fulfilled.

A family of examples that fulfill Assumption~\ref{ass:Sregular} can easily be constructed by letting $B: \mathbb{R} \rightarrow H$ be given
by
\begin{equation*}
  Bu = u \rho
\end{equation*}
for any $\rho \in W \subset H$. The  adjoint $B^*: H \rightarrow  \mathbb{R} $ is given by
\begin{equation*}
  B^*v = \iprod{v, \rho}_H,
\end{equation*}
so $S = BB^* : H \rightarrow H$ satisfies
\begin{equation*}
  Sv = BB^*v = \iprod{\rho,v}_H \rho.
\end{equation*}
This means that $S$ actually maps into $W$. By choosing $W = \domain{A} = H^2 \cap H^1_{\text{per}}(\Omega)$, Assumption~\ref{ass:Sregular} is satisfied. Using the same construction but instead choosing $\rho \in H \setminus \domain{A}$ means that the assumption will not be satisfied.
Since $A = A^*$ in our experiments, the above procedure can also be used to construct a $P_0$ that either satisfies Assumption~\ref{ass:P0regular} or not, by simply replacing $B$ with $G^*$.

We generate the functions $\rho_1$ and $\rho_2$ in  $H^2 \cap H^1_{\text{per}}(\Omega)$ and $\xi_1$ and $\xi_2$ in $H^0(\Omega) = H$ by sampling two different Q-Wiener processes, as described in~\cite[Section 10.2]{LordPowellShardlow2014}.
Essentially, this means that we compute weighted sums of eigenfunctions to the Laplacian, where the weights are appropriately scaled random numbers. The functions are initially evaluated on a grid with $N_h = 2^{17}$ grid points and then projected onto each of the coarser grids. We choose different combinations of $\rho_j$ and $\xi_j$, $j = 1,2$, to construct $S$ and $P_0$ as described above, such that all combinations of fulfilling Assumption~\ref{ass:P0regular} and Assumption~\ref{ass:Sregular} are covered. These choices are listed in Table~\ref{tab:experiment_overview}, and the functions themselves are plotted in Figure~\ref{fig:rho_and_xi}.

\begin{remark}
  The operators $S = 0$ and $P_0 = 0$ clearly satisfy the assumptions. However, $S = 0$ is not a good choice for our purposes here, since it completely removes the nonlinearity in the DRE  and thus means that the splitting schemes are exact. Similarly, $P_0 = 0$ leads to an exact solution $P$ which is $\CHk{\infty}{[0,T]}$. This follows from a minor modification to the proof of Theorem IV-1.3.1 in~\cite{Bensoussan2007}. While it has not been rigorously proven in this paper, we expect that such high regularity will lead to optimal convergence orders for both schemes, regardless of any assumption on $S$. That is, Assumption~\ref{ass:Sregular} would not need to be satisfied for the Strang splitting to be of order $2$ when $P_0 = 0$. This is why we do not use this initial condition in the experiments and instead consider a more general one.
\end{remark}
\begin{table}
  \centering
\begin{tabular}{r|cccc}
  Experiment & $S$ & $P_0$ & Assumption~\ref{ass:P0regular} & Assumption~\ref{ass:Sregular} \\
  \hline{}
  1 & $\rho_1$ & $\rho_2$ & Yes & Yes \\
  2 & $\rho_1$ & $\xi_2$ & No & Yes \\
  3 & $\xi_1$ & $\rho_2$ & Yes & No \\
  4 & $\xi_1$ & $\xi_2$ & No & No 
\end{tabular}
\caption{Overview of the experiment setup.}
\label{tab:experiment_overview}
\end{table}

\begin{figure}
  \centering
  \includegraphics[width=\textwidth]{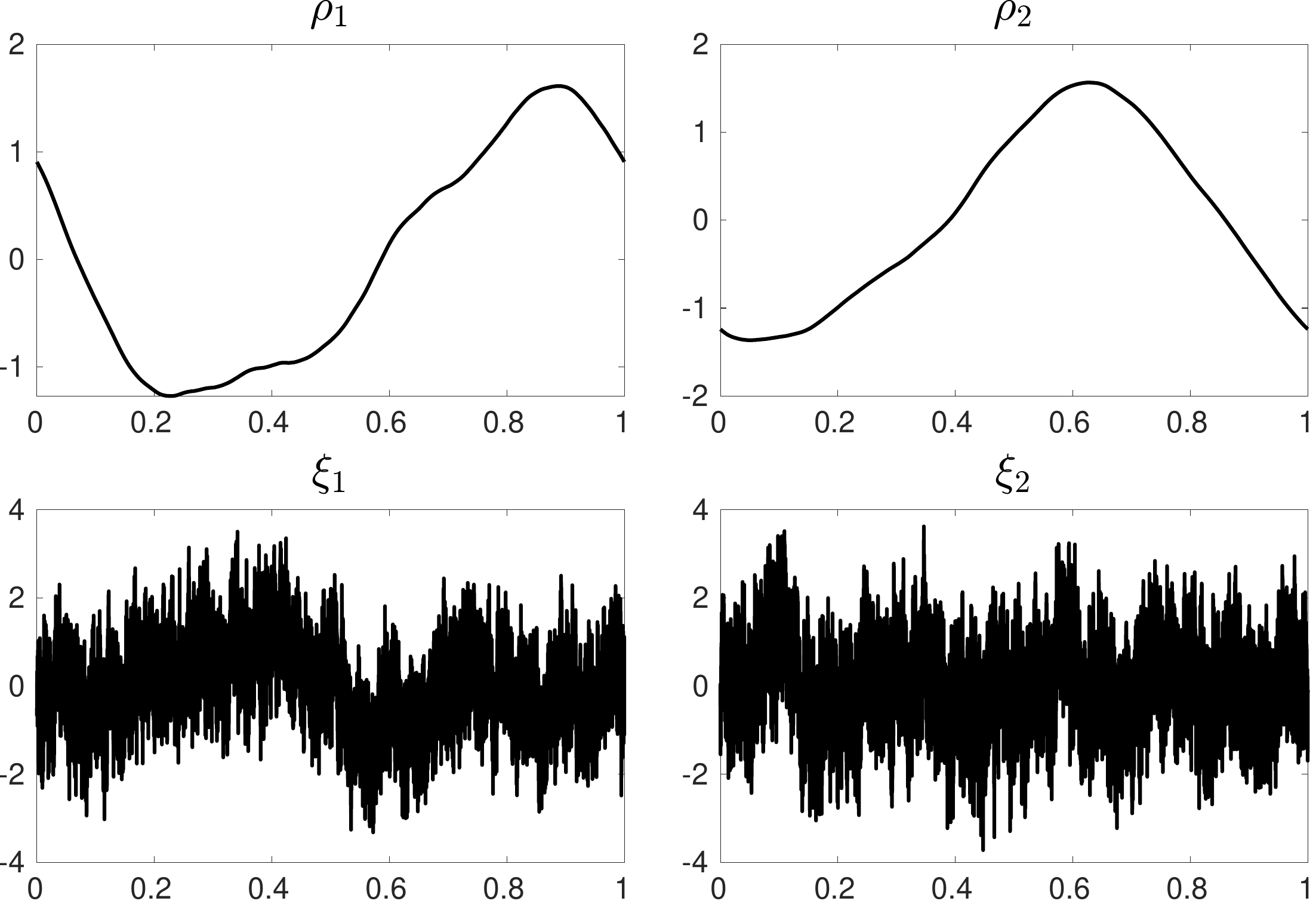}
  \caption{The functions $\rho_j \in H^2 \cap H^1_{\text{per}}(\Omega)$ and $\xi_j \in H$, $j = 1, 2$ used to generate the operators $P_0$ and $S$ in the experiments.}
  \label{fig:rho_and_xi}
\end{figure}

Since we cannot time-step the operator-valued DREs directly, we first discretize
in space using the finite element method (FEM) with piecewise linear basis functions $\phi_j$, $j=1, \ldots, N_h$. We denote the approximation space by $V_h := \operatorname{span}\big(\{\phi_j\}_{j=1}^{N_h}\big)$ and note that $V_h \subset V = H^1_{\text{per}}(\Omega) \subset H$. As described in detail in~\cite{MalqvistPerssonStillfjord2018}, the operator-valued DRE~\eqref{eq:DREclassical} then turns into the matrix-valued DRE
\begin{equation} \label{eq:DRE_matrix}
\Mb \dot{\Pb} \Mb = \Mb \Pb \Ab + \Ab^T \Pb \Mb + \Eb^T \Eb - \Mb \Pb \Bb \Bb^T \Pb \Mb,
\end{equation}
where the matrices $\Ab \in \R^{N_h \times N_h}$, $\Mb \in \R^{N_h \times N_h}$, $\Bb \in \R^{N_h \times 1}$, $\Eb \in \R^{1 \times N_h}$ are given by
$\Ab_{i,j} = \iprod{A\phi_j, \phi_i}$,
$\Mb_{i,j} = \iprod{\phi_j, \phi_i}$,
$\Bb_{i,j} = \iprod{B 1, \phi_i}$,
$\Eb_{i,j} = \iprod{E\phi_j, 1}$.
The matrix $\Pb(t) \in \R^{N_h \times N_h}$ is the matrix representation of the FEM-approximation $P_h(t) : V_h \to V_h$ to $P(t) : H \to H$. If $V_h \ni x = \sum_{j=1}^{N_h}{\xb_j \phi^h_j}$ with $\xb \in \R^{N_h}$ then
\begin{equation} \label{eq:Ph_formula2}
  P_h(t) x = \sum_{i,j=1}^N{\Pb_{i,j}(t)\iprod{x,\phi_j} \phi_i}.
\end{equation}
The initial value $P_h(0)$ is the projection of $P(0)$ onto $\cL(V_h, V_h)$, which satisfies
$P_h(0) = (\Id_h)^* P(0) \Id_h$ where $\Id_h : V_h \to V \subset H$ is the identity operator and $(\Id_h)^* : H \to V_h$ is the $H$-orthogonal projection onto $V_h$.

The natural extension of $P_h(t)$ to an operator on $H$ is given by $\Id_h P_h(t) (\Id_h)^*$, and the $\cL(H)$-norm of this operator is given by
\begin{equation*}
  \normLHH{\Id_h P_h(t) (\Id_h)^*} = \norm{\Lb_{\Mb}^T \Pb(t) \Lb_{\Mb}}_{\R^{N \times N}},
\end{equation*}
where $\Mb = \Lb_{\Mb} \Lb_{\Mb}^T$ is a Cholesky factorization of the mass matrix and $\norm{\cdot}_{\R^{N \times N}}$ denotes the standard spectral matrix norm.

In each experiment, we use several different spatial discretizations, corresponding to $N_h = 2^k$, $k = 2, \ldots, 14$. For each $N_h$, we apply the Lie and Strang splitting methods with $N_t = 2^j$ time steps for $j = 2, \ldots, 12$, corresponding to time steps of size $\tau = 2^{-j}$, until reaching $T = 0.1$. We use the low-rank-factored implementations of these schemes provided in the Matlab package DRESplit\footnote{Available from the corresponding author on request, or from \url{www.tonystillfjord.net}.} with the low-rank tolerance set to $10^{-15}$.

Denote the approximation of $P_h(n\tau)$ by $P_h^n$. The temporal errors for each $h$ are computed as
\begin{equation*}
  \textnormal{err}_{\tau} = \maxnorm{\Id_h (P_h^{\cdot} - P_{\tau,{\text{ref}}}^{\cdot})(\Id_h)^*}
\end{equation*}
where the reference approximation $P_{\tau,{\text{ref}}}$ is given by applying the Strang splitting approximation with time step $\tau_{\text{ref}} = 2^{-13}$ and discarding the intermediate steps which do not coincide with $n\tau$, $n = 1, \ldots, N_t$.

As $N_h$ increases, the spatial discretization gets finer, and the matrix-valued DRE tends to the corresponding operator-valued DRE as $N_h \to \infty$. For each fixed spatial discretization, we expect the time discretizations to converge with order $1$ and $2$, respectively. But when $N_h \to \infty$, this should break down, unless a relevant combination of Assumptions~\ref{ass:P0regular} and~\ref{ass:Sregular} are fulfilled. One way this can manifest in is that the largest $\tau$ for which the ``correct'' orders are observed decreases as $N_h \to \infty$.

\subsection{Experiment 1}
The errors are shown in Figure~\ref{fig:exp1}, and we observe temporal convergence with perfect order $1$ and $2$ for the Lie and Strang splittings, respectively, regardless of the spatial discretization. Since both Assumptions~\ref{ass:P0regular} and~\ref{ass:Sregular} are fulfilled, this is the expected result according to Theorems~\ref{thm:Lie1Bound}, \ref{thm:Lie2Bound} and~\ref{thm:StrangBound}.
\begin{figure}
  \centering
  \includegraphics[width=\textwidth]{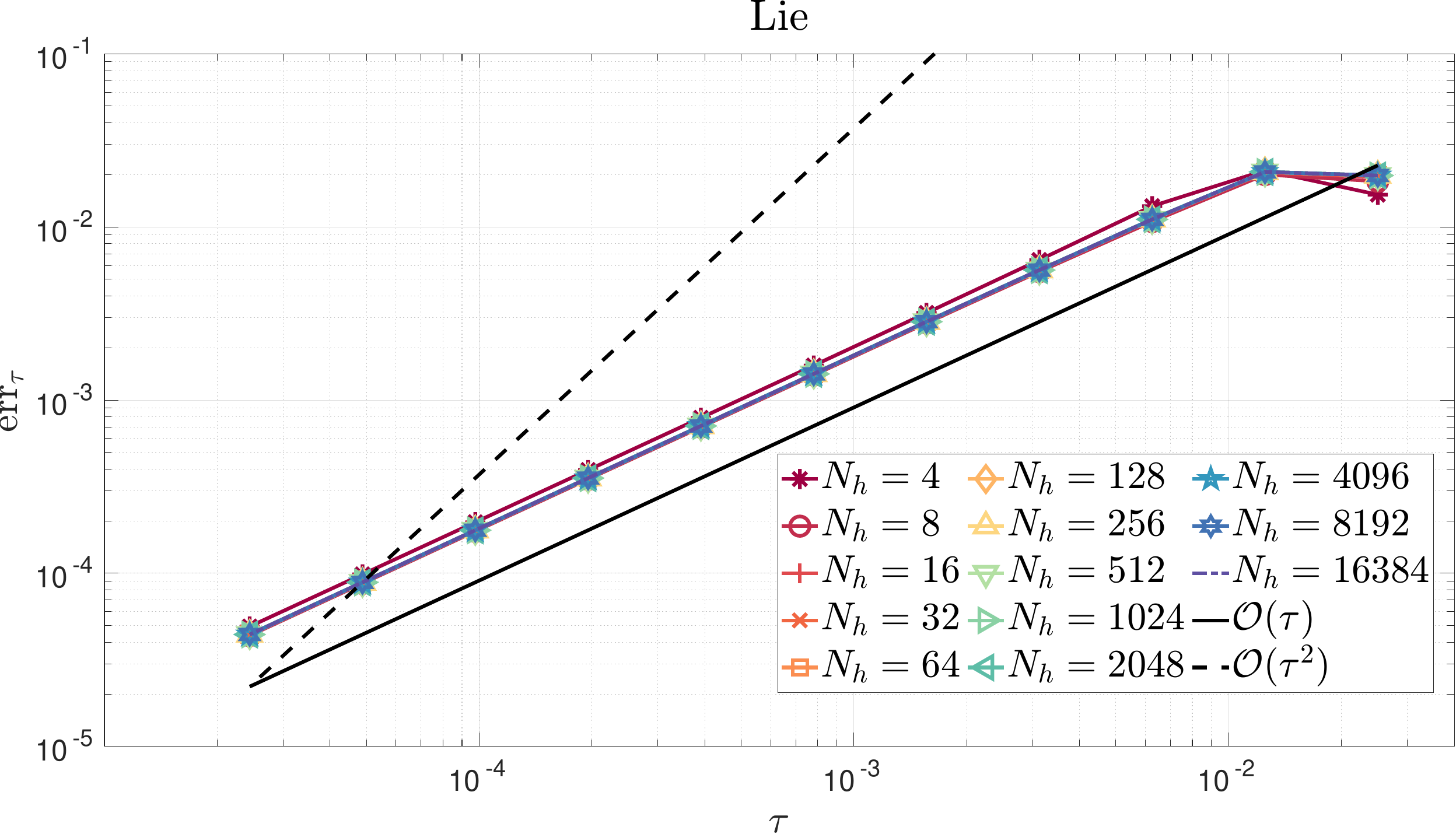}
  \includegraphics[width=\textwidth]{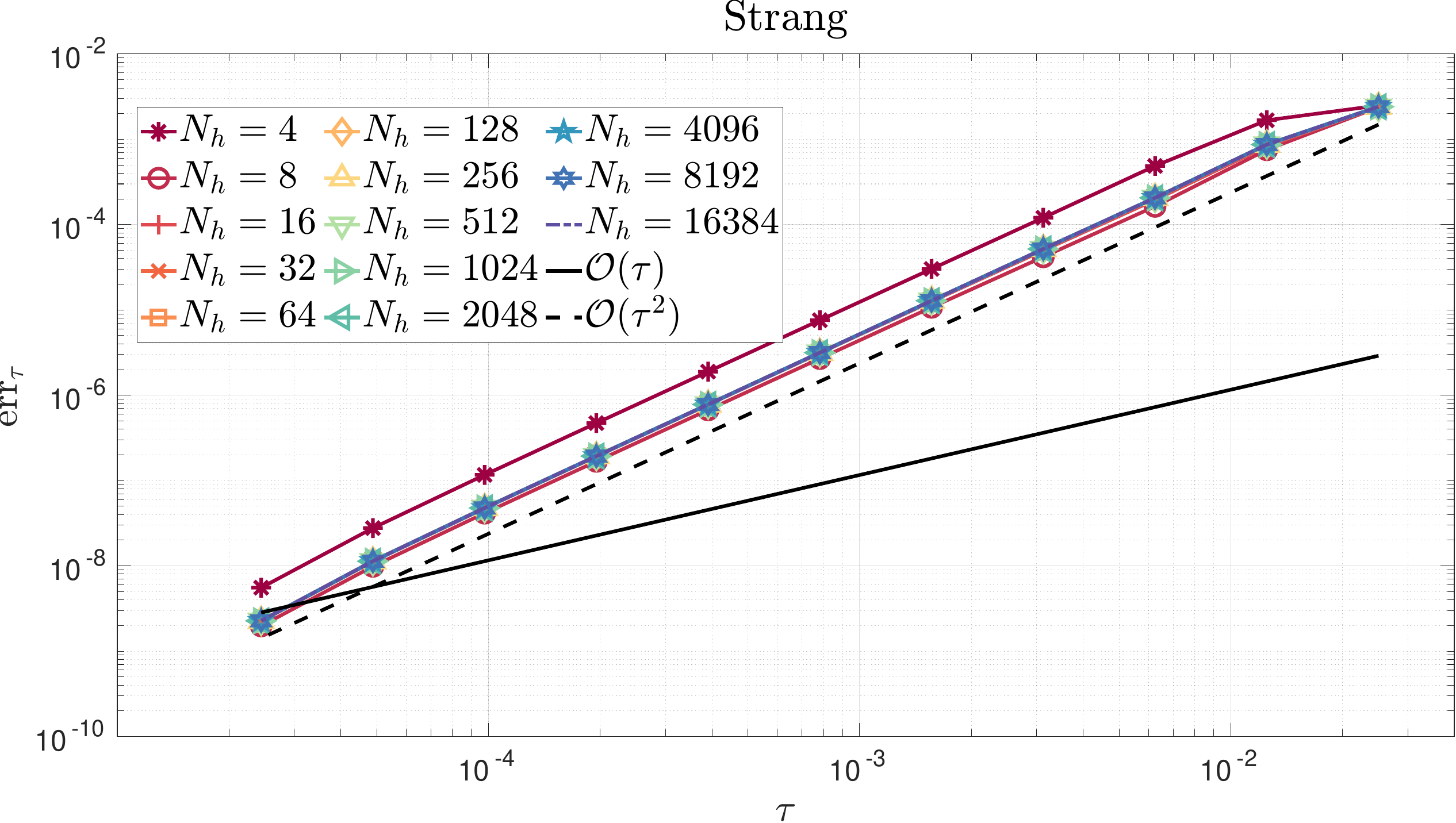}
  \caption{The errors $\textnormal{err}_{\tau}$ for Lie (top) and Strang (bottom) splitting in a setting when Assumptions~\ref{ass:P0regular} and~\ref{ass:Sregular} are both fulfilled. Each curve corresponds to a different spatial discretization with $N_h$ nodes.}
  \label{fig:exp1}
\end{figure}

\subsection{Experiment 2}
The errors are shown in Figure~\ref{fig:exp2}, and we again observe temporal convergence with essentially perfect order $1$ for the Lie splitting, regardless of the spatial discretization. This is in line with Theorem~\ref{thm:Lie2Bound}. For the Strang splitting, however, only the coarsest spatial discretization, corresponding to $N_h = 4$, yields a curve of slope $2$ for the whole interval of time steps. For finer spatial discretizations, the errors behave more like $\cO(\tau)$ until $\tau$ is small enough. Since only one of the assumptions required for Theorem~\ref{thm:StrangBound} is satisfied, this is the expected result. We also note that, in contrast to Experiment 1, the errors initially increase as $N_h$ increases. It is unclear why this increase stops around $N_h = 128$, but we conjecture that is is related to the fact that each spatial discretization employs its own reference approximation.

\begin{figure}
  \centering
  \includegraphics[width=\textwidth]{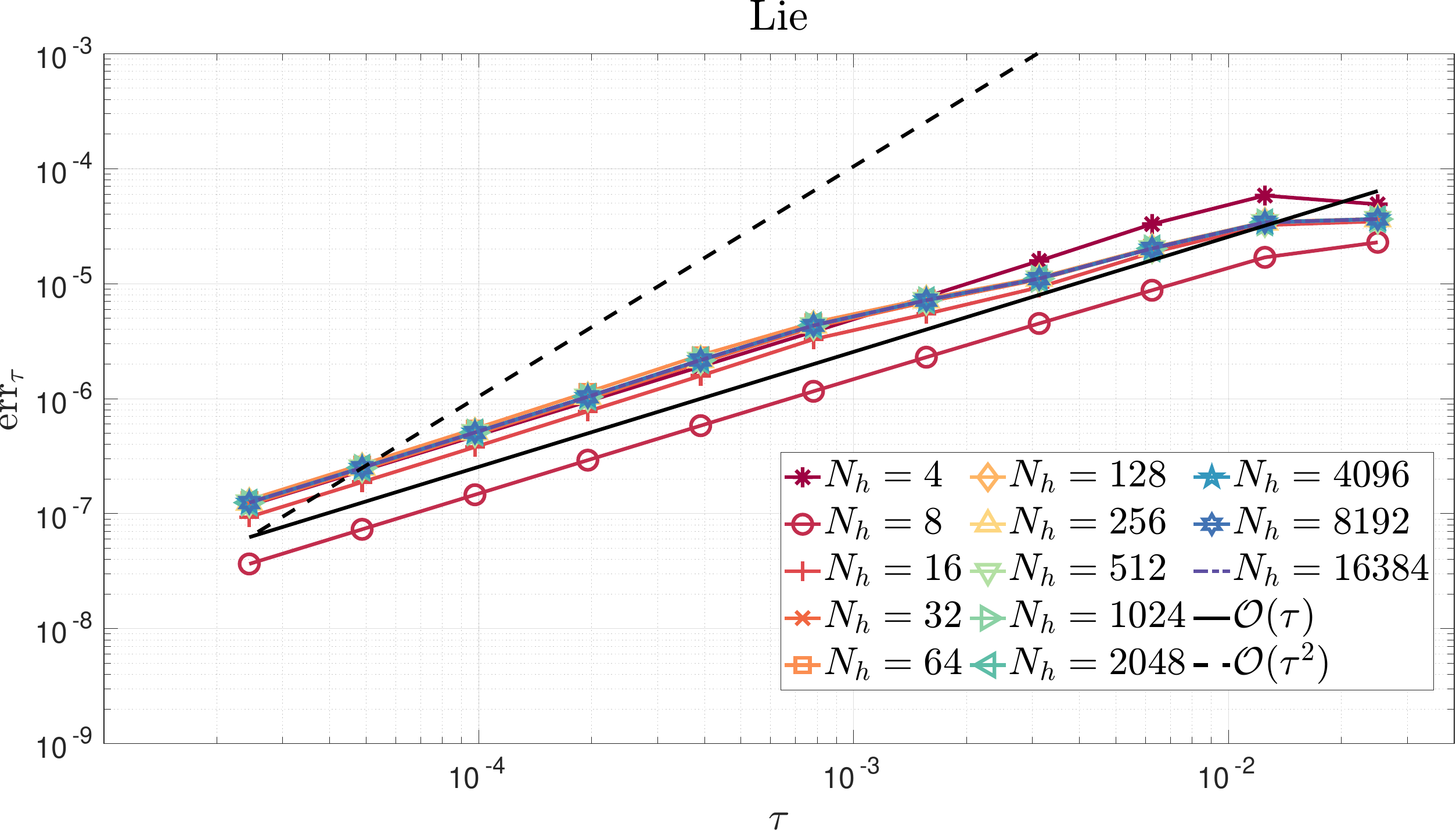}
  \includegraphics[width=\textwidth]{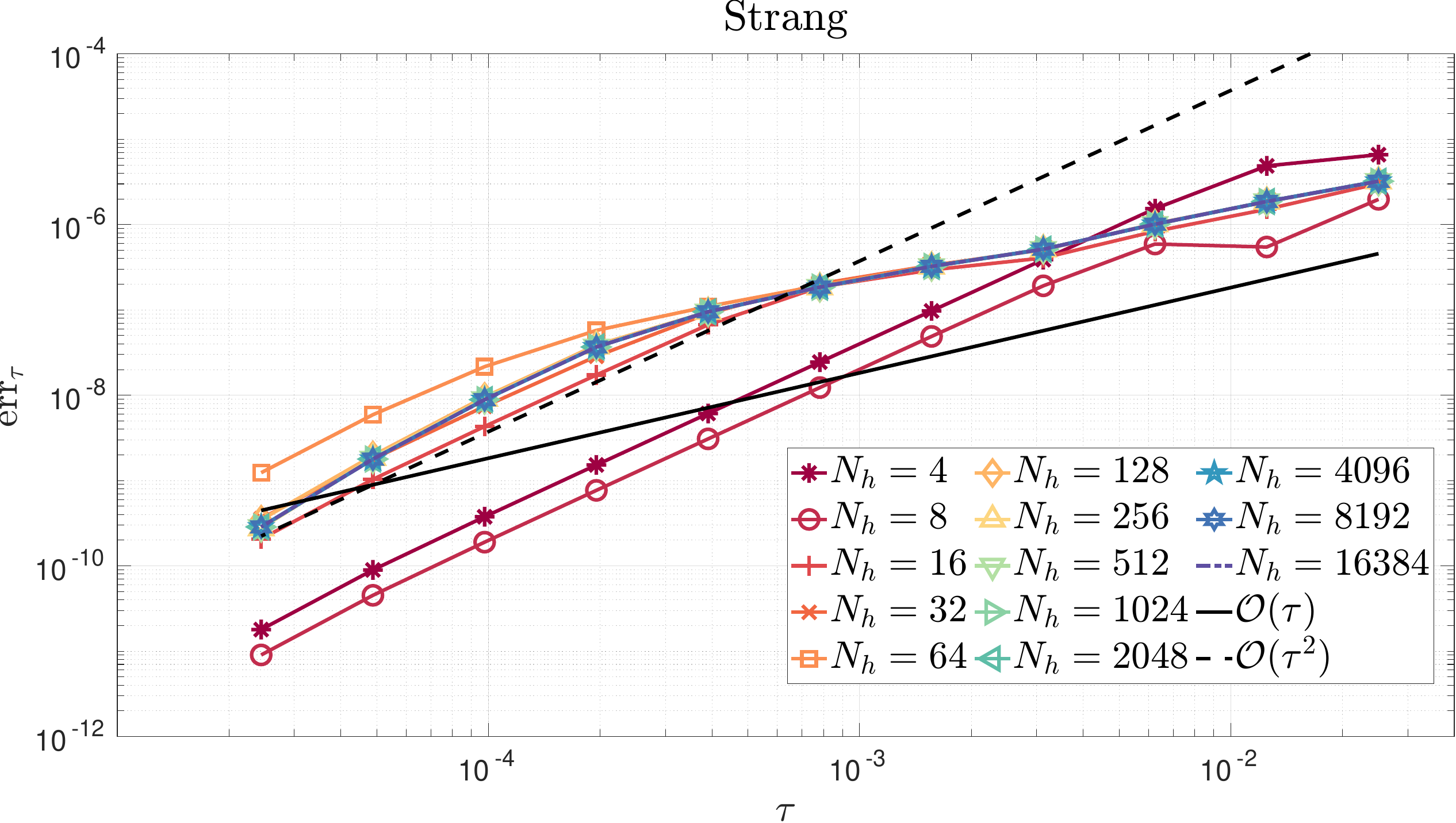}
  \caption{The errors $\textnormal{err}_{\tau}$ for Lie (top) and Strang (bottom) splitting in a setting when Assumption~\ref{ass:Sregular} is fulfilled but Assumption~\ref{ass:P0regular} is not. Each curve corresponds to a different spatial discretization with $N_h$ nodes.}
  \label{fig:exp2}
\end{figure}

\subsection{Experiment 3}
The errors in this experiment are shown in Figure~\ref{fig:exp3}, and they are quite similar to those in Experiment 1. Since Assumption~\ref{ass:P0regular} is satisfied, Theorem~\ref{thm:Lie1Bound} explains the Lie results. However, as Assumption~\ref{ass:Sregular} is not fulfilled, Theorem~\ref{thm:StrangBound} cannot be used to argue for why the Strang splitting performs so well. With this said, we have not proven that the given assumptions are necessary conditions, only sufficient. Better than expected convergence properties can occur in some cases, but this cannot be relied upon to always happen.

\begin{figure}
  \centering
  \includegraphics[width=\textwidth]{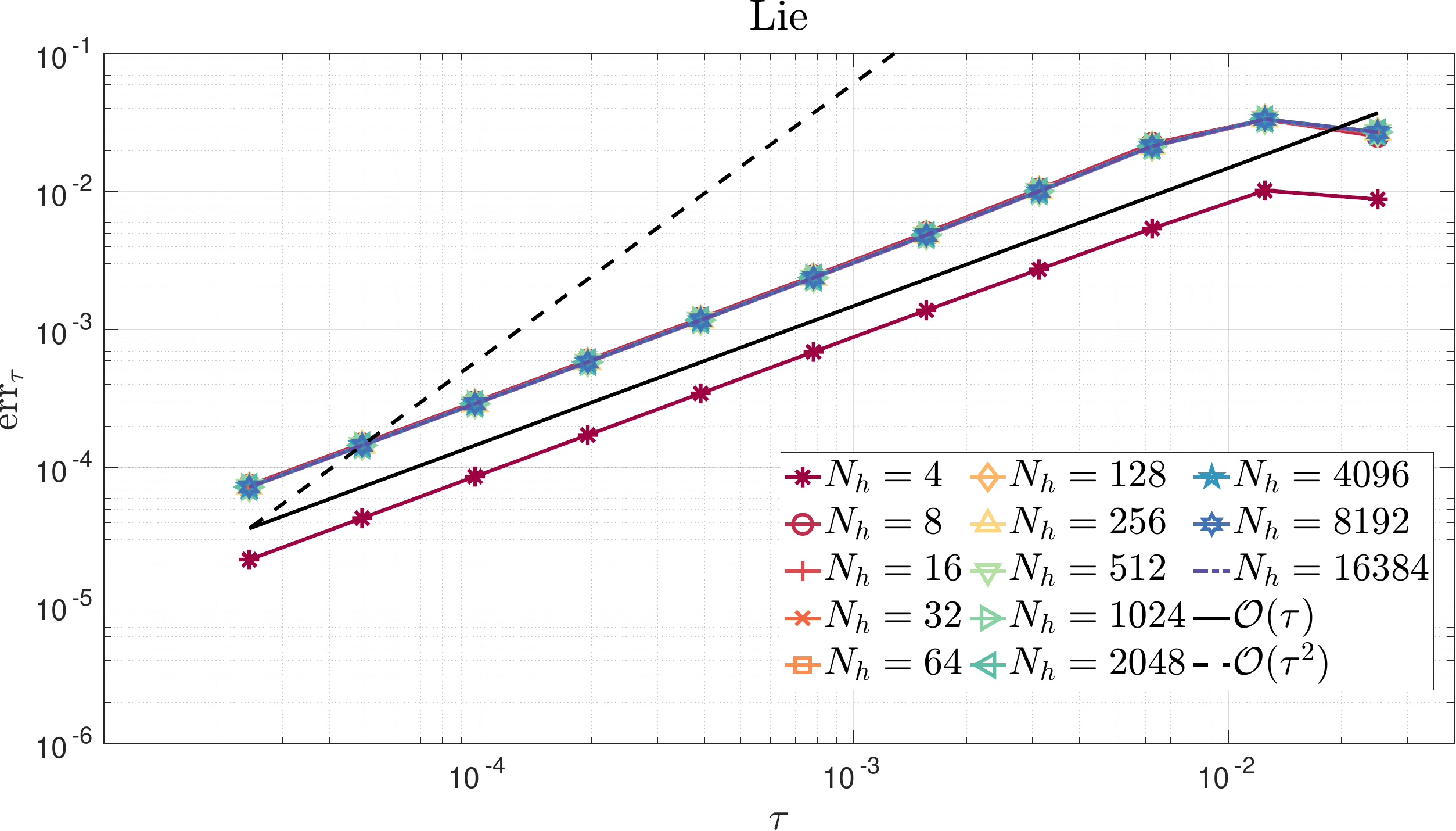}
  \includegraphics[width=\textwidth]{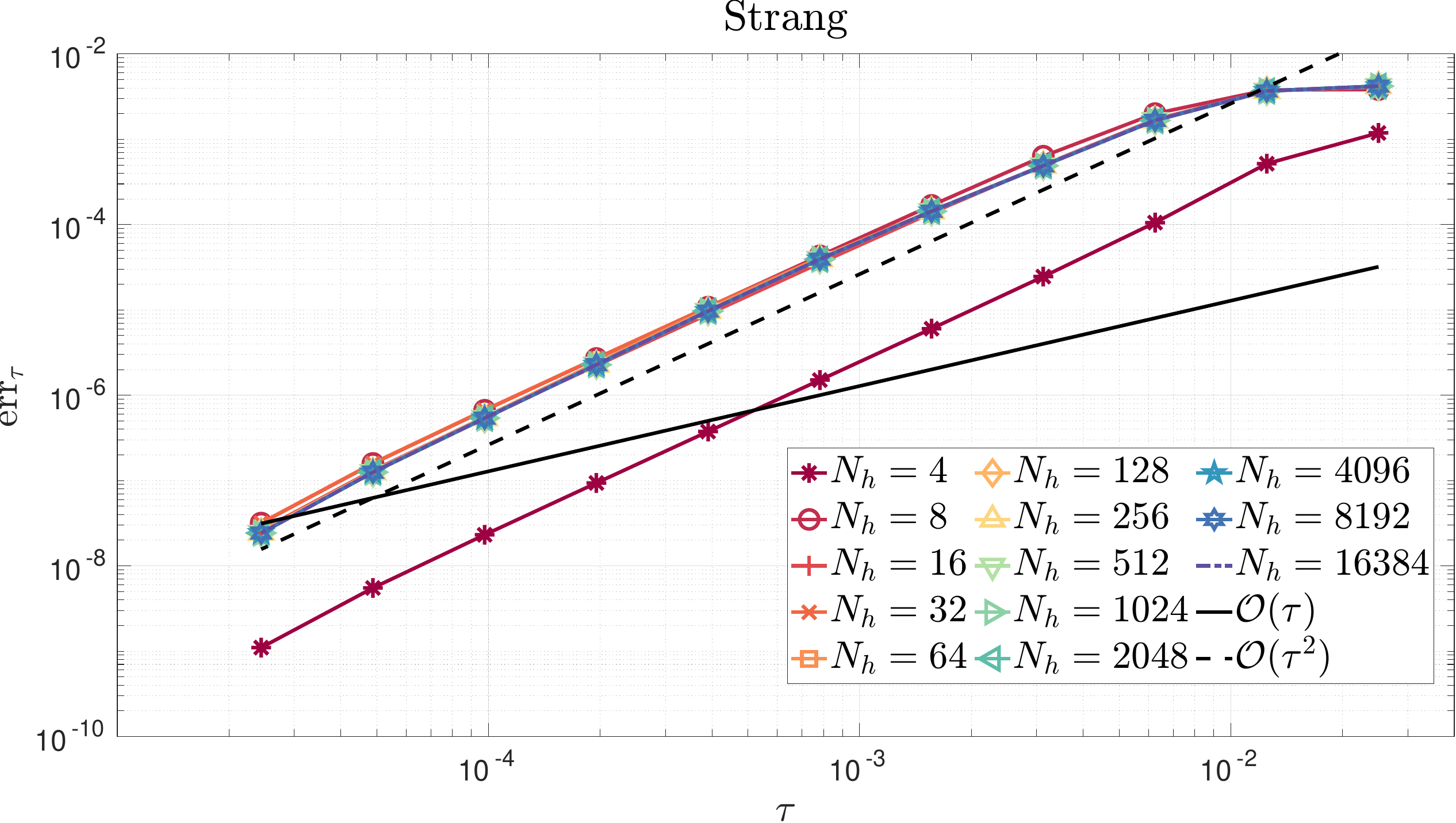}
    \caption{The errors $\textnormal{err}_{\tau}$ for Lie (top) and Strang (bottom) splitting in a setting when Assumption~\ref{ass:P0regular} is fulfilled but Assumption~\ref{ass:Sregular} is not. Each curve corresponds to a different spatial discretization with $N_h$ nodes.}
    \label{fig:exp3}
  \end{figure}

\subsection{Experiment 4}
In this final experiment, none of the assumptions are fulfilled, and the errors shown in Figure~\ref{fig:exp4} accordingly behave more erratically than in the other experiments. For Lie splitting, we observe convergence of order 1 for the coarsest spatial discretizations, but as $N_h$ increases this decays and the errors increase. The effect is seen even more clearly for Strang splitting, where convergence of order 2 is initially observed, but for each increase of $N_h$, the time step size after which this is observed becomes smaller. We also see more clearly how the errors increase as $N_h$ increases.

\begin{figure}
  \centering
  \includegraphics[width=\textwidth]{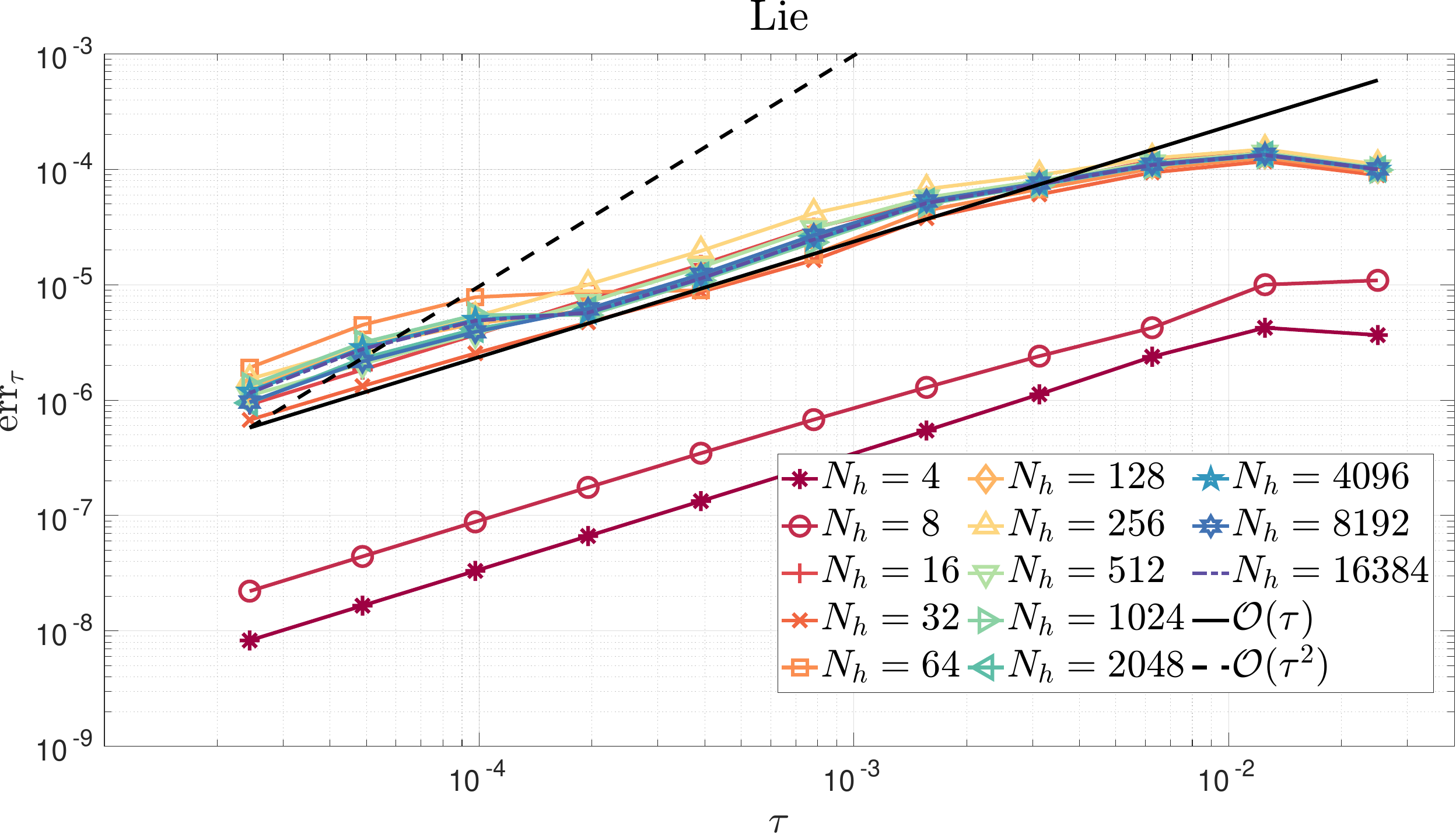}
  \includegraphics[width=\textwidth]{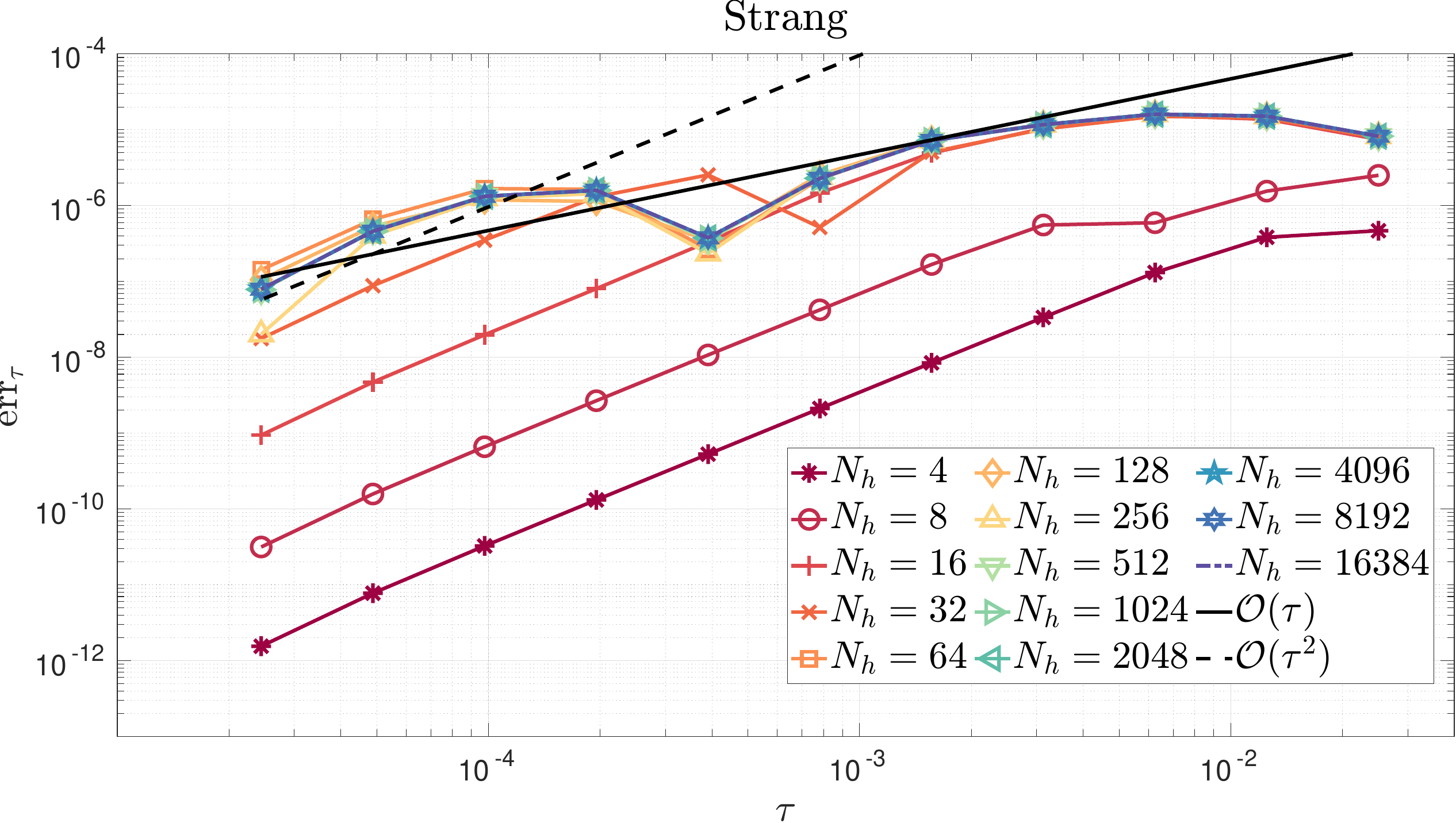}
  \caption{The errors $\textnormal{err}_{\tau}$ for Lie (top) and Strang (bottom) splitting in a setting when neither Assumption~\ref{ass:P0regular} nor Assumption~\ref{ass:Sregular} is fulfilled. Each curve corresponds to a different spatial discretization with $N_h$ nodes.}
  \label{fig:exp4}
  \end{figure}

\section{Conclusions}\label{sec:conclusions}
We have rigorously proved first- and second-order convergence of Lie and Strang splitting when applied to operator-valued DRE. These are the classical orders, which are to be expected for matrix-valued DREs under essentially no assumptions, since those are equivalent to systems of ODEs. However, as confirmed by the numerical experiments, if those matrix-valued equations are spatial discretizations of operator-valued equations then additional assumptions are necessary to avoid order reduction when the discretizations are refined. Our analysis relies on Assumption~\ref{ass:P0regular}  or Assumption~\ref{ass:Sregular}, either of which yield the requisite regularity of the exact solution. While we have not proven this, it seems unlikely that weaker assumptions would result in the full classical orders. However, we conjecture that assuming, e.g., $P_0 \in \domain{A^r}$ with $0 < r < 1$ should lead to a solution $P \in \CpHk{r}{[0,T]}$ and thence to convergence with order $r$ for Lie splitting. Investigating such modifications of the error analysis, as well as extensions to the case when $B$ and $E$ in the corresponding LQR problem are relatively unbounded operators, are topics for future work.

\section*{Acknowledgments}
The authors were partially supported by the Swedish Research Council under grants 2023-03982 (EH, TS, TÅ) and 2023-04862 (EH, TS).

\bibliography{refs}
\bibliographystyle{abbrvnat}

\end{document}